\documentclass[leqno,final]{siamltex}

\usepackage{graphicx}
\usepackage{epstopdf, color}
\usepackage{amsmath,amstext,amssymb,hyperref}
\allowdisplaybreaks[2]

\numberwithin{equation}{section}

\newcommand{\norm}[1]{\left\Vert#1\right\Vert}

\newcommand{\norme}[1]{\left\Vert{\hskip -2.7pt}\left\vert #1 \right\vert{\hskip -2.7pt}\right\Vert}

\newcommand{\abs}[1]{\left\vert#1\right\vert}
\newcommand{\pd}[1]{\left\langle #1\right\rangle}

\newcommand{\set}[1]{\left\{#1\right\}}

\newcommand{\jm}[1]{\left[#1\right]}

\newcommand{\R}{\mathbb{R}}
\newcommand{\tuh}{\tilde{u}_h}

\newcommand{\nn}{\nonumber}
\newcommand{\ls}{\lesssim}

\newcommand{\al}{\alpha}
\newcommand{\be}{\beta}

\newcommand{\De}{\Delta}
\newcommand{\ep}{\varepsilon}
\newcommand{\ga}{\gamma}
\newcommand{\Ga}{\Gamma}
\newcommand{\la}{\lambda}

\newcommand{\na}{\nabla}
\newcommand{\om}{\omega}
\newcommand{\Om}{\Omega}
\newcommand{\pa}{\partial}

\newcommand{\ta}{\theta}

\newcommand{\dx}{\,\mathrm{d} x}

\newcommand{\N}{\mathcal{N}}
\newcommand{\cH}{\mathcal{H}}

\renewcommand{\i}{{\rm\mathbf i}}
\newcommand{\bt}{{\rm\mathbf t}}
\newcommand{\bn}{{\rm\mathbf n}}
\newcommand{\bm}{{\rm\mathbf m}}

\newcommand{\T}{\mathcal{T}}
\newcommand{\E}{\mathcal{E}}

\newcommand{\eq}[1]{\begin{align}#1\end{align}}
\newcommand{\eqn}[1]{\begin{align*}#1\end{align*}}

\title{superconvergence analysis of DG-FEM based on the polynomial preserving recovery for\\ Helmholtz equation with high wave number}
\author{
Yu Du\footnotemark[1]
\thanks{Beijing computational science research center, Beijing, 100193, China. {\tt duyu87@csrc.ac.cn, dynju@qq.com}.
This research work is supported by a Tianhe--2JK computing time award at the Beijing Computational Science Research Center (CSRC).
The research of this author was supported in part by the China Postdoctoral Science Foundation under grant 2016M591053
and the National Natural Science Foundation of China under grants 11601026}
\and
Zhimin Zhang\footnotemark[2]
\thanks{Beijing Computational Science Research Center, Beijing, 100193 and Department of Mathematics, Wayne State University, Detroit, MI 48202. {\tt zmzhang@csrc.ac.cn, zzhang@math.wayne.edu}. The research of this author was supported in part by the National Natural Science Foundation of China under grants 11471031, 91430216, U1530401, and the U.S. National Science Foundation through grant DMS--1419040.}
}

\begin{document}

\maketitle

\vspace{-1.4in}
\slugger{sinum}{200x}{xx}{x}{xxx--xxx}
\vspace{1.4in}

\setcounter{page}{1}

\begin{abstract} We study superconvergence property
of the linear discontinuous Galerkin finite element method with the polynomial preserving recovery (PPR)
and Richardson extrapolation
for the two dimensional Helmholtz equation. The error estimate
with explicit dependence on the wave number $k$, the penalty parameter $\mu$ and the mesh condition parameter $\al$ is derived.
 First, we prove that under the assumption $k(kh)^2\leq C_0$ ($h$ is the mesh size) and certain mesh condition,
 the estimate between the finite element solution and the linear interpolation of
the exact solution is superconvergent under the $\norme{\cdot}$-seminorm.
Second, we prove a superconvergence result for the recovered gradient by PPR.
Furthermore, we estimate the error between the finite element gradient and recovered gradient, which motivate us to define the a posteriori error estimator.
Finally, Some numerical examples are provided to confirm the theoretical results of superconvergence analysis.
All theoretical findings are verified by numerical tests.
\end{abstract}

\begin{keywords}
Helmholtz equation, large wave number, pollution errors,
superconvergence, polynomial preserving recovery, discontinuous Galerkin finite element methods
\end{keywords}

\begin{AMS}
65N12, 
65N15, 
65N30, 
78A40  
\end{AMS}

\setcounter{page}{1}

\section{Introduction}\label{intro}
This paper is devoted to superconvergence analysis of the linear Galerkin finite element method for the following Helmholtz problem:
\eq{
-\De u - k^2 u &=f  \qquad\mbox{in  } \Om,\label{eq1.1a}\\
\frac{\pa u}{\pa n} +\i k u &=g \qquad\mbox{on } \Ga,\label{eq1.1b}
}
where $\Om\in\R^2$ is a bounded polygon with boundary $\Ga:=\pa\Om$, $\i=\sqrt{-1}$ denotes the imaginary unit and $n$ denotes the unit outward normal to $\Ga$.
The above Helmholtz problem is an approximation of the acoustic scattering problem (with time dependence $e^{\i\om t}$) and $k$ is known as the wave number. The Robin boundary condition
\eqref{eq1.1b} is known as the first order approximation of the radiation condition (cf. \cite{em79}).
The Helmholtz problem \eqref{eq1.1a}--\eqref{eq1.1b} also arises in applications as a consequence of frequency domain
treatment of attenuated scalar waves (cf. \cite{dss94}).

It is well known that the finite element method of fixed order for the Helmholtz
problem \eqref{eq1.1a}--\eqref{eq1.1b} at high frequencies ($k\gg1$) is subject to the effect of pollution:
the ratio of the error of the finite element solution to the error of the best approximation
from the finite element space cannot be uniformly bounded with respect to $k$
\cite{Ainsworth04,bs00,bips95,dbb99,harari97,ib95a,ib97}.
More precisely, the linear finite element method for a $2$-D Helmholtz problem
satisfies the following error estimate under the mesh constraint $k(kh)^2\leq C_0$ \cite{zw,dw}:
\eq{ \norm{\na(u-u_h)}_{L^2(\Om)} \leq C_1kh+C_2k(kh)^2. \label{int-eq1}}
Here $u_h$ is the linear finite element solution, $h$ is the mesh size and $C_i,i=1,2$ are positive constants independent of $k$ and $h$.
It is easy to see that the order of the first term on the right hand side of \eqref{int-eq1} is the same to that
of the interpolation error in $H^1$-seminorm and it can dominate the error bound only if $k(kh)$ is small
enough. However, the second term on the {right-hand side} of \eqref{int-eq1} dominates the estimate under other mesh conditions.
For example, $kh$ is fixed and $k$ is large enough. The term $C_2k(kh)^2$ is called the pollution error of
the finite element solution.

Considerable efforts have been made in analysis of different numerical methods
for the Helmholtz problem with large wave number in the literature.
The readers are referred to \cite{ak79,dss94,sch74} for asymptotic error estimates of general DG methods and
\cite{ib95a,ib97} for pre-asymptotic error estimates of a one-dimensional problem discretized on equidistant grid.
For more pre-asymptotic error estimates, Please refer to \cite{ms10,ms11}
and \cite{zbw,zw} for classical finite element methods as well as interior penalty finite element methods.
For other methods solving the Helmholtz problems, such as the interior
penalty discontinuous Galerkin method or the source transfer domain decomposition method,
one can read \cite{mps13,fw09,fw11,zd,dzh,cx}.

In this work, we investigate the superconvergence property of the linear discontinuous Galerkin (DG) finite element method
when being post-processed by our polynomial preserving recovery (PPR) for the Helmholtz problem.
PPR was proposed by Zhang and Naga \cite{zn05} in 2004 and has been successfully applied to finite element methods.
COMSOL Multiphysics adopted PPR as a post-processing tool since 2008 \cite{comsol}.
\cite{dwz} has applied the technique to the Helmholtz problem and prove its superconvergence property.
In this paper, we generalize the technique over the DG finite element space and prove its superconvergence property.
To learn more about PPR, readers are referred to \cite{z04,z04t,nz04,wz07}.
Some theoretical results about recovery techniques and recovery-type error
estimators can be found in \cite{bx03,lmw,zl99,xz03,yz01}.


Our purpose of this paper is to prove the superconvergence error estimates for the linear discontinuous Galerkin finite element method and analyze the influence of the PPR technique on the pollution error. Note the superconvergence error estimates depend on the triangulation, the penalty parameter and the wave number under certain mesh condition. In order to prove the estimates, we first assume some mesh constraints, called ``Condition $\al$'', and then redefine the PPR method on the discontinuous Galerkin finite element space. Finally, all the estimates motivate us to combine the PPR technique and the Richardson extrapolation to reduce the error further and define the a posterior error estimator.

The remainder of this paper is organized as follows: some notations, DG-FEM and the mesh constraints are
introduced in section \ref{pre}. In section \ref{sup}, we prove the superconvergence between the interpolant and
the finite element solution to the problem with Robin boundary \eqref{eq1.1a}--\eqref{eq1.1b}. In section \ref{sec_ppr},
we redefine the PPR technique over the linear DG finite element space and  prove the superconvergence property of $G_h$ in the Sobolev space $H^3$ and show the most important result,
that is the error estimate of $G_hu_h$. Then we try to give the reason for the effect of $G_h$ to the pollution error in
section \ref{sec_ppr_nabla}. Finally, we simulate a model problem by the linear DG-FEM, PPR method and the Richardson extrapolation in section \ref{sec_num}.
It is shown that the recovered gradient can be improved by the Richardson extrapolation further and the a posterior error estimator based on the PPR and Richardson extrapolation is exact asymptotically.

Throughout the paper, we assume that $\Om$ is a strictly star-shaped domain.
Here ``strictly star-shaped'' means that there exist a point $x_\Om\in\Om$ and a positive
constant $c_\Om$ depending only on $\Om$ such that
\eqn{ (x-x_\Om)\cdot n\geq c_\Om\quad \forall x\in\Ga. }
\section{Preliminaries} \label{pre}
Throughout this paper, we assume that for any node point $z\in\N_h$, there exists at least one interior edge $e\in\E_h^I$ having $z$.

To introduce the method and simplify the analysis, we introduce some notation first. The standard Sobolev and Hilbert space, norm, and inner product notation
are adopted (cf. \cite{bs08,ciarlet78}). In particular, $(\cdot,\cdot)_Q$ and $\pd{\cdot,\cdot}_\Sigma$
for $\Sigma=\pa Q$ denote the $L^2$-inner product on complex-valued $L^2(Q)$ and $L^2(\Sigma)$ spaces, respectively.

Let $\T_h$ be a regular triangulation of the domain $\Om$. For any $\tau\in\T_h$, we denote by $h_\tau$ its diameter and by $\abs{\tau}$ its area. Let $h=\max_{\tau\in\T_h}h_\tau$. Assume that $h_\tau\eqsim h$.

Let $V_h$ be the approximation space of piecewise linear polynomials, that is,
\eqn{ V_h:=\set{v_h\in L^2(\Om): v_h|_\tau\in P_1(\tau)\quad \forall \tau\in\T_h}, }
where $P_1(\tau)$ denotes the set of all polynomials on $\tau$ with degree $\leq1$.

$\E_h$ be the set of all edges of $\T_h$ and $\N_h$ be the set of all nodal points. Denote all the boundary edges by $\E_h^B:=\set{e\in\E_h:e\subset\Ga}$ and the interior edges by $\E_h^I:=\E_h\backslash\E_h^B$.
For each edge $e\in\E_h$, define $h_e:={\rm diam}(e)$. For $e=\pa \tau\cap\tau'\in\E_h^I$, let $\bn_e$ be a unit normal
vector to $e$. We assume that the normal vector $\bn_e$ is oriented from $\tau'$ to $\tau$ and define
\eqn{ [v]|_e:=v|_{\tau'}-v|_\tau,\ \set{v}:=\frac{1}{2}(v|_{\tau'}+v|_\tau). }

We define the space $E:=\prod_{\tau\in\T_h}H^1(\tau)$ and introduce the sesquilinear form $a_h(\cdot,\cdot)$ on $E\times E$ as follows:
\eqn{ a_h(u,v) &:= \sum_{K\in\T_h} (\na u,\na v)_K - \sum_{e\in\E_h^I} \left( \pd{\set{\frac{\pa u}{\pa \bn_e}},\jm{v}}_e + \pd{\jm{u},\set{\frac{\pa v}{\pa\bn_e}}}_e \right) \\
 &\qquad + J_0(u,v) \\
J_0(u,v) &:= \sum_{e\in\E_h^I} \frac{\rho_{0,e}}{h_e^{1+\mu}} \pd{\jm{u},\jm{v}}_e, \\
 }
The linear DG method is defined as follows: find $u_h\in V_h$ such that
\eq{ a_h(u_h,v_h) - k^2(u_h,v_h) + \i k\pd{u_h,v_h} = (f,v_h) + \pd{g,v_h}\quad \forall v_h\in V_h. \label{eq_DG-FEM}}

\emph{Remark 2.1}
(a) Note that the method is the standard symmetric DG method (cf. \cite{br}). So we have the proposition (cf. proposition 2.9 in \cite{br}):
$u\in H^2(\Om)$ is the solution to \eqref{eq1.1a}--\eqref{eq1.1b} if and only if $u$ satisfies the general DG variational formulation
\eqn{ a_h(u,v) - k^2(u,v) + \i k\pd{u,v} = (f,v) + \pd{g,v}\quad \forall v\in E. }

(b) A similar DG method, called interior penalty discontinuous Galerkin method (IPDG), was introduced and analyzed by Feng, Wu and so on for the Helmholtz problem \eqref{eq1.1a}--\eqref{eq1.1b}. The reader is referred to \cite{fw09,fw11,zd,dzh} for both asymptotic and preasymptotic error estimates.

The following two norms will be used in the forthcoming sections :
\eqn{ \norm{v}_{1,h} &:= \left( \sum_{K\in\T_h}\norm{v}_{L^2(K)}^2 + J_0(v,v) \right)^{1/2}\quad \forall v\in E,\\
\norme{v}_{1,h} &:= \left( \norm{v}_{1,h}^2 + k^2\norm{v}_0^2 \right)^{1/2}\quad \forall v\in E. }
We denote by $\abs{\cdot}_{\cH^1(\omega)}$ the broken semi-$H^1$ norm $\left(\sum_{\tau\in\omega}\abs{\cdot}_{H^1(\tau)}^2\right)^{1/2}$ where $\omega\subset\T_h$.

Throughout the paper�� $C$ denotes a generic positive constant
which is independent of $h,k,f,g$ and the penalty parameters. We use the shorthand notation $A\ls B$ for
$A\leq CB$ and assume $k\gg1$ since we are considering high-frequency problems.
We assume that the solution $u$ to the problem \eqref{eq1.1a}--\eqref{eq1.1b} is $H^3$-regular over $\Om$ and the data $g$
is $H^2$-regular over $\Ga$. Denote by
\eq{C_{u,g}=\sum_{j=1}^3 k^{-(j-1)}\norm{u}_j + \sum_{j=1}^2 k^{-j}\abs{g}_{H^j(\Ga)}.}

The function $C_{u,g}$ could be treated as a constant in this paper since $\norm{u}_j$
is bounded by $\max(k^0,k^{j-1})$ when $u$ is the solution to the Helmholtz problem \eqref{eq1.1a}--\eqref{eq1.1b}. The reader is referred to \cite{mps13,ms10,ms11} for the estimates of $u$.

Before estimating the errors, we state the coercivity and continuity properties for the sesquilinear form $a_h(\cdot,\cdot)$. Since they easily follow from the difinitions 2.10--2.11 in \cite{br}, the details are omitted.
\begin{lemma}\label{lemma_norms_equal}
For any $v,w\in E$, the sesquilinear form $a_h(\cdot,\cdot)$ satisfies
\eqn{ \abs{a_h(v,w)},\abs{a_h(w,v)} \ls \norm{w}_{1,h}\norm{v}_{1,h}. }
In addition, there exists constant $\underline{\rho}$ such that if $\underline{\rho}\leq\rho_{0,e}$,
\eqn{ &\norm{v_h}_{1,h}^2 \ls a_h(v_h,v_h)\qquad \forall v_h\in V_h,\\
& \norme{v_h}_{1,h}^2 \ls a_h(v_h,v_h) + k^2(v_h,v_h)\qquad \forall v_h\in V_h. }
\end{lemma}
The following lemma shows the preasymptotic error estimates for the solution $u$ to \eqref{eq1.1a}--\eqref{eq1.1b}.
The results can be derived by arguments same to those in \cite{dzh,zd} and we omit the details to save space.
\begin{lemma}\label{lemma_pre_err}
Assume that $u$ is the solution the problem \eqref{eq1.1a}--\eqref{eq1.1b} and $u_h$ is the discrete solution of the scheme \eqref{eq_DG-FEM}. Then there exists constant $C_0$ independent of $k$ and $h$, such that if $k(kh)^2\le C_0$ then the following estimates hold:
\begin{align}
\norm{u-u_h}_{1,h}&\ls \big(kh+k(kh)^{2}\big)C_{u,g},\label{ecor-1-a}\\
k\|u-u_h\|_0&\ls \big((kh)^{2}+k(kh)^{2}\big)C_{u,g}.\label{ecor-1-b}
\end{align}
\end{lemma}

We begin with some definitions regarding meshes. 
For an interior edge $e\in\E_h^I$, we denote $\Om_e=\tau_e\cup\tau_e'$,
a patch formed by the two elements $\tau_e$ and $\tau_e'$ sharing $e$, see Figures~\ref{figNot}-\ref{figNotb}.
For any edge $e\in\E_h$ and an element $\tau$ {with} $e\subset\tau$,
$\ta_e$ denotes the angle opposite of the edge $e$ in $\tau${,
$\bt_e$ denotes the unit tangent vector of $e$ with counterclockwise orientation and $\bn_e$, the unit outward normal vector of $e$,
$h_e, h_{e+1}$, and $h_{e-1}$ denote the lengths of the three edges of $\tau$, respectively.
Here the subscript $e+1$ or $e-1$ is for orientation.
Note that all triangles in the triangulation are orientated counterclockwise,
and the} index $'$ is added for the corresponding quantities in $\tau'$ with $\bt_e=-\bt_e'$ and $\bn_e=-\bn_e'$
due to the orientation.

For any $e\in\E_h^I$ (cf. Figure~\ref{figNot}), we say that $\Om_e$ is an $\ep$ approximate parallelogram if the
lengths of any two opposite edges differ by at most $\ep$, that is,
\eqn{ \abs{h_{e-1}-h_{e-1}'} + \abs{h_{e+1}-h_{e+1}'}\le\ep. }

For any $e\in\E_h^B$ (cf. Figure~\ref{figNotb}), we say that $\tau_e$ is an $\ep$ approximate isosceles triangle if the lengths of
its two edges $e-1$ and $e+1$ differ by at most $\ep$, that is,
\eqn{ \abs{h_{e+1}-h_{e-1}}\le\ep. }

\begin{definition} \label{meshcond}
The triangulation $\T_h$ is said to satisfy \emph{mesh condition $\al$} if there exists a constant $\al\geq0$
such that
\begin{itemize}
  \item[(a)] the patch $\Om_e$ is an $O(h^{1+\al})$ approximate parallelogram for any interior edge $e\in\E_h^I$;
  \item[(b)] the triangle $\tau_e$ is an $O(h^{1+\al})$ approximate isosceles triangle for any boundary edge $e\in\E_h^B$;
\end{itemize}
\end{definition}

\emph{Remark} 2.1.
The restriction $(\al)$ in Definition~\ref{meshcond} is often used
to prove the superconvergence property for problems with the Dirichlet boundary condition \cite{cx07,wz07}.
Note that this restriction is technique and just for theoretical
purpose. In fact, one superconvergence results still can be obtained under general meshes which do not satisfy the condition, such as Delaunay triangulation and Chevron pattern triangulation.

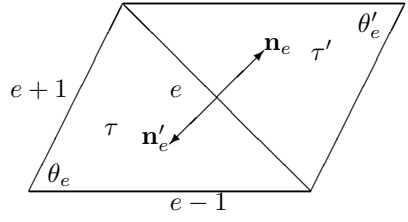
\begin{figure}
\begin{center}
\setlength{\unitlength}{0.25mm}
  \begin{picture}(390,110)


  \linethickness{0.25mm}
  \put(110,10){\line(1,0){150}}
  \put(160,110){\line(1,0){150}}
  \put(110,10){\line(1,2){50}}
  \put(260,10){\line(1,2){50}}
  \put(260,10){\line(-1,1){100}}

  \put(185,60){$e$}
  \put(150,40){$\tau$}
  \put(260,80){$\tau'$}
  \put(100,60){$e+1$}
  \put(185,0){$e-1$}
  \put(120,15){$\ta_e$}
  \put(285,95){$\ta_e'$}

  \put(210,60){\vector(-1,-1){25}}
  \put(210,60){\vector(1,1){25}}
  \put(170,35){$\bn_e'$}
  \put(235,85){$\bn_e$}
  \end{picture}
\end{center}
\caption{Notation in the patch $\Om_e$.}
\label{figNot}
\end{figure}

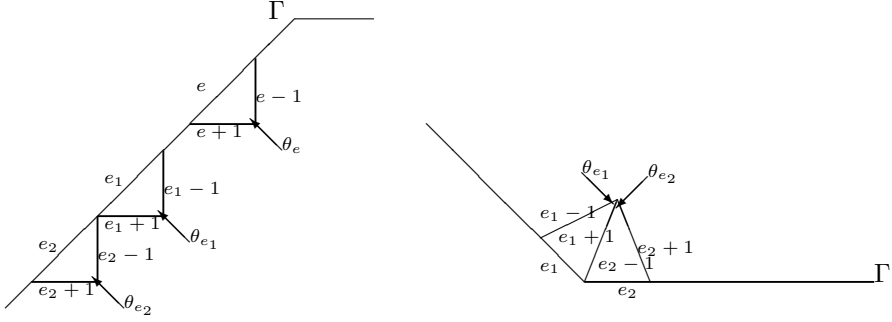
\begin{figure}
\begin{center}
\setlength{\unitlength}{0.35mm}
  \begin{picture}(390,110)

  \put(120,110){\line(1,0){30}}
  \put(10,0){\line(1,1){110}}
  \put(110,110){$\Ga$}

  \linethickness{0.25mm}
  \put(20,10){\line(1,0){25}}
  \put(45,10){\line(0,1){25}}
  \put(22.5,22.5){\scriptsize $e_2$}
  \put(45,18.5){\scriptsize $e_2-1$}
  \put(22.5,5){\scriptsize $e_2+1$}
  \put(55,0){\vector(-1,1){12}}
  \put(55,0){\scriptsize $\ta_{e_2}$}


  \put(45,35){\line(1,0){25}}
  \put(70,35){\line(0,1){25}}
  \put(47.5,47.5){\scriptsize $e_1$}
  \put(70,43.5){\scriptsize $e_1-1$}
  \put(47.5,30){\scriptsize $e_1+1$}
  \put(80,25){\vector(-1,1){12}}
  \put(80,25){\scriptsize $\ta_{e_1}$}

  \put(80,70){\line(1,0){25}}
  \put(105,70){\line(0,1){25}}
  \put(82.5,82.5){\scriptsize $e$}
  \put(105,78.5){\scriptsize $e-1$}
  \put(82.5,65){\scriptsize $e+1$}
  \put(115,60){\vector(-1,1){12}}
  \put(115,60){\scriptsize $\ta_{e}$}

  \put(230,10){\line(1,0){110}}
  \put(230,10){\line(-1,1){60}}
  \put(340,10){$\Ga$}

  \put(230,10){\line(2,5){12.5}}
  \put(255,10){\line(-2,5){12.5}}
  \put(242.5,5){\scriptsize $e_2$}
  \put(250,21){\scriptsize $e_2+1$}
  \put(235,15){\scriptsize $e_2-1$}
  \put(254,50){\vector(-1,-1){12}}
  \put(254,50){\scriptsize $\ta_{e_2}$}

  \put(213.3,26.7){\line(2,1){29}}
  \put(213,14.3){\scriptsize $e_1$}
  \put(213,34.3){\scriptsize $e_1-1$}
  \put(220,25){\scriptsize $e_1+1$}
  \put(229,51.5){\vector(1,-1){12}}
  \put(229,51.5){\scriptsize $\ta_{e_1}$}

  \end{picture}
\end{center}
\caption{Notation in the boundary elements.}
\label{figNotb}
\end{figure}

\section{ Superconvergence between the discontinuous finite element solution and linear interpolant} \label{sup}
First we introduce a quadratic interpolant { $\psi_Q=\Pi_Q\psi$ of $\psi$ based on nodal values and}
moment conditions on edges,
\eq{ (\Pi_Q\phi)(z)=\phi(z),\quad \int_e\Pi_Q\phi=\int_e\phi\quad \forall z\in\N_h,e\in\E_h. }

The following fundamental identity for $v_h\in P_1(\tau)$ has been proved in \cite{cx07}:
\eq{ \int_\tau \na(\phi-\phi_I)\cdot\na v_h = \sum_{e\in\pa\tau} \bigg( \be_e\int_e\frac{\pa^2\phi_Q}{\pa \bt_e^2}\frac{\pa v_h}{\pa \bt_e} + \ga_e\int_e\frac{\pa^2\phi_Q}{\pa \bt_e\pa \bn_e}\frac{\pa v_h}{\pa \bt_e} \bigg) \label{fdmtliequ} }
where
\eq{ \be_e=\frac{1}{12}\cot\ta_e(h_{e+1}^2-h_{e-1}^2),\quad \ga_e=\frac{1}{3}\cot\ta_e\abs{\tau}, }
and $\phi_I\in P_1(\tau)$ is the linear interpolant of $\phi$ on $\tau$. The following lemma can be easily
obtained \cite{cx07,wz07}.
\begin{lemma} \label{parsestm}
Let $\bm_e$ denote $\bt_e$ or $\bn_e$. Assume that $\T_h$ satisfies { the} \emph{mesh condition $\al$}, then
we have the following estimates:
\begin{itemize}
  \item[(a)] For any interior edge $e\in\E_h^I$,
  \eq{ &\abs{\be_e}+\abs{\be_e'}\ls h^2,\quad \abs{\ga_e}+\abs{\ga_e'}\ls h^{2};\label{beta_gamma1}\\
  &\abs{\be_e-\be_e'}\ls h^{2+\al},\quad \abs{\ga_e-\ga_e'}\ls h^{2+\al}.\label{beta_gamma2}}
  \item[(b)] For two adjacent edges $e_1,e_2\in\E_h^B$, that is $e_1\cap e_2\neq\emptyset$,
  \eq{ &\abs{\be_{e_1}}+\abs{\be_{e_2}}\ls h^{2+\al},\quad \abs{\ga_{e_1}}+\abs{\ga_{e_2}}\ls h^{2}�� \label{beta_gamma3}\\
  & \abs{\ga_{e_1}-\ga_{e_2}}\ls h^{2+\al}\label{beta_gamma4}}
  \item[(c)] For any edge $e\in\E_h$, $e\subset\pa\tau_e$,
  \eq{  &\int_e \frac{\pa^2\phi}{\pa \bt_e\pa \bm_e}\frac{\pa v_h}{\pa \bt_e} \ls (\norm{\phi}_{H^{3}(\tau_e)}+h^{-1}\norm{\phi}_{H^{2}(\tau_e)})\abs{v_h}_{\cH^1(\tau_e)};\label{trcinq1}\\
  &\int_e \frac{\pa^2(\phi-\phi_Q)}{\pa \bt_e\pa \bm_e}\frac{\pa v_h}{\pa \bt_e} \ls \abs{\phi}_{H^{3}(\tau_e)}\abs{v_h}_{\cH^1(\tau_e).\label{trcinq2}}
}
\end{itemize}
\end{lemma}
\begin{proof}
The inequalities \eqref{beta_gamma1}--\eqref{beta_gamma3} follow from the \emph{mesh condition $\al$}. From the condition (a) and (b) in Definition~\ref{meshcond}, we have
for any $e_1, e_2\in\E_h^B$ satisfying $e_1\cap e_2\neq\emptyset$ (cf. Figure~\ref{figNotb}),
\eqn{ \frac{\big|h_{e_1-1}h_{e_1+1}\cos\ta_{e_1}-h_{e_2-1}h_{e_2+1}\cos\ta_{e_2}\big|}{h} \ls h^{1+\al}, }
which implies \eqref{beta_gamma4}.

Finally, the inequalities \eqref{trcinq1} and \eqref{trcinq2} follow from the trace theorem.
\end{proof}

\begin{lemma} \label{uui}
Assume that $\T_h$ satisfies {the} \emph{$\al$ approximation ondition}. Then for any $v_h\in V_h$,
\eq{ \abs{\sum_{K\in\T_h}\int_K\na(u-u_I)\cdot\na v_h}\ls \left((kh)^2+kh^{1+\al}+kh^{1+\mu/2}\right)\norme{v_h}_{1,h} C_{u,g}. }
Here $u_I$ is the { linear} interpolant of $u$ on $\Om$.
\end{lemma}
\begin{proof}
From \eqref{fdmtliequ}, we have
\eqn{
\int_\Om\na(u-u_I)\cdot\na v_h &= \sum_{\tau\in\T_h}\sum_{e\subset\pa\tau} \left( \be_e\int_e\frac{\pa^2 u_Q}{\pa \bt_e^2}\frac{\pa v_h}{\pa \bt_e} + \ga_e\int_e\frac{\pa^2 u_Q}{\pa \bt_e \pa \bn_e}\frac{\pa v_h}{\pa \bt_e} \right)\\
&= I_1+I_2,
}
where
\eqn{
&I_1 = \sum_{e\in\E_h^I} \left[ \int_e\frac{\pa^2 u}{\pa \bt_e^2} \bigg({\be_e\frac{\pa v_h}{\pa \bt_e}|_{K_1}-\be_e'\frac{\pa v_h}{\pa \bt_e}|_{K_2}}\bigg) + \int_e\frac{\pa^2 u}{\pa \bt_e \pa \bn_e} \bigg(\ga_e\frac{\pa v_h}{\pa \bt_e}|_{K_1}-\ga_e'\frac{\pa v_h}{\pa \bt_e}|_{K_2}\bigg) \right. \\
&\qquad\qquad + \be_e\int_e\frac{\pa^2 (u_Q-u)}{\pa \bt_e^2}\frac{\pa v_h}{\pa \bt_e} + \ga_e \int_e\frac{\pa^2 (u_Q-u)}{\pa \bt_e \pa \bn_e}\frac{\pa v_h}{\pa \bt_e}\\
&\qquad\qquad \left. + \be_e'\int_e\frac{\pa^2 (u-u_Q)}{\pa \bt_e^2}\frac{\pa v_h}{\pa \bt_e} + \ga_e' \int_e\frac{\pa^2 (u-u_Q)}{\pa \bt_e \pa \bn_e}\frac{\pa v_h}{\pa \bt_e} \right],\\
&\quad = I_{1,1} + I_{1,2} + I_{1,3}, \\
&I_2 = \sum_{e\in\E_h^B} \left[ \be_e\int_e\frac{\pa^2 u}{\pa \bt_e^2}\frac{\pa v_h}{\pa \bt_e} + \ga_e\int_e\frac{\pa^2 u}{\pa \bt_e \pa \bn_e}\frac{\pa v_h}{\pa \bt_e} \right. \\
&\qquad\qquad \left. + \be_e\int_e\frac{\pa^2 (u_Q-u)}{\pa \bt_e^2}\frac{\pa v_h}{\pa \bt_e} + \ga_e \int_e\frac{\pa^2 (u_Q-u)}{\pa \bt_e \pa \bn_e}\frac{\pa v_h}{\pa \bt_e}\right].
}

First, $I_{1,2}$ and $I_{1,3}$ can be estimated by Lemma~\ref{parsestm} and H\"{o}lder's inequality:
\eqn{
\abs{I_{1,2}+I_{1,3}} &\ls \sum_{e\in\E_h^I} \left( (h^{2+\al}+h^2)\norm{u}_{H^{3}(\tau_e)} + h^{1+\al}\norm{u}_{H^{2}(\tau_e)} \right) \abs{v_h}_{\cH^1(\tau_e)} \\
&\ls \left( (h^{2+\al}+h^2)\norm{u}_3 + h^{1+\al}\norm{u}_2 \right) \abs{v_h}_{\cH^1} \\
&\ls \left( (kh)^{2} + kh^{1+\al} \right)\abs{v_h}_{\cH^1}  C_{u,g}.
}
From Lemma~\ref{parsestm} and the inverse inequality,
\eqn{ I_{1,1} &=  \sum_{e\in\E_h^I} \left[ (\be_e-\be_e') \int_e\frac{\pa^2 u}{\pa \bt_e^2} \frac{\pa v_h}{\pa \bt_e}|_{K_1} + (\ga_e-\ga_e')\int_e\frac{\pa^2 u}{\pa \bt_e \pa \bn_e} \frac{\pa v_h}{\pa \bt_e}|_{K_1} \right. \\
&\qquad + \left. \be_e'\int_e\frac{\pa^2 u}{\pa \bt_e^2} \jm{\frac{\pa v_h}{\pa \bt_e}}+ \ga_e \int_e\frac{\pa^2 u}{\pa \bt_e \pa \bn_e} \jm{\frac{\pa v_h}{\pa \bt_e}} \right] \\
&\ls \sum_{e\in\E_h^I} \bigg[ \left(h^{2+\al}\norm{u}_{H^{3}(\tau_e)} + h^{1+\al}\norm{u}_{H^{2}(\tau_e)}\right) \abs{v_h}_{\cH^1(\tau_e)} \\
&\qquad  + \left(h^{5/2}\norm{u}_{H^{3}(\tau_e)} + h^{3/2}\norm{u}_{H^{2}(\tau_e)}\right) \bigg(\sum_{e\in\E_h^I}\norm{\jm{\frac{\pa v_h}{\pa \bt_e}}}_{L^2(e)}^2\bigg)^{1/2} \bigg] \\
&\ls \left( (kh)^{2} + kh^{1+\al} +kh^{1+\mu/2} \right)\norm{v_h}_{1,h}  C_{u,g}, }
where we have used the ineqaulities
\eqn{ \bigg(\sum_{e\in\E_h^I}\norm{\jm{\frac{\pa v_h}{\pa \bt_e}}}_{L^2(e)}^2\bigg)^{1/2} &= \bigg(\sum_{e\in\E_h^I}\norm{\frac{\pa \jm{v_h}}{\pa \bt_e}}_{L^2(e)}^2\bigg)^{1/2} \ls \bigg(\sum_{e\in\E_h^I}h_e^{-2}\norm{\jm{v_h}}_{L^2(e)}^2\bigg)^{1/2} \\
&\ls h^{\mu/2-1/2}J_0(v_h,v_h)^{1/2} \leq h^{\mu/2-1/2} \norm{v_h}_{1,h}. }
By combining the inequalities above, we get
\eq{ \abs{I_1}\ls \left( (kh)^{2} + kh^{1+\al} +kh^{1+\mu/2} \right)\norm{v_h}_{1,h}  C_{u,g}. \label{I1est}}
Then we turn to the estimate of $I_2$. From \eqref{beta_gamma3} and \eqref{trcinq2},
\eq{
&\sum_{e\in\E_h^B} \left[ \be_e\int_e\frac{\pa^2 u}{\pa t_e^2}\frac{\pa v_h}{\pa t_e} + \be_e\int_e\frac{\pa^2 (u_Q-u)}{\pa t_e^2}\frac{\pa v_h}{\pa t_e} + \ga_e \int_e\frac{\pa^2 (u_Q-u)}{\pa t_e \pa n_e}\frac{\pa v_h}{\pa t_e}\right]\label{part1_I2}\\
&\ls \sum_{e\in\E_h^B} \big( h^{1+\al}\norm{u}_{H^2(\tau_e)}+(h^{2+\al}+h^2)\norm{u}_{H^3(\tau_e)} \big) \abs{v_h}_{\cH^1(\tau_e)}\nn\\
&\ls \big( kh^{1+\al}+(kh)^{2} \big)\abs{v_h}_{\cH^1}  C_{u,g}.\nn
}
Therefore, we only need to estimate the remaining terms of $I_2$. Denote by $z_i$ the nodes on $\Ga$. Denote by $e_{i},e_{i+1}\in\E_h^B$ sharing $z_i$ with counterclockwise orientation (cf. Figure~\ref{figNotb}). Denote by $\jm{\cdot}_{z_i}=(\cdot|_{e_{i+1}})(z_i)-(\cdot|_{e_i})(z_i)$ and $\set{\cdot}_{z_i}=\big(\cdot|_{e_i}(z_i)+\cdot|_{e_{i+1}}(z_i)\big)/2$. Then we have
\eq{
\sum_{e\in\E_h^B}\ga_e\int_e\frac{\pa^2 u}{\pa t_e \pa n_e}\frac{\pa v_h}{\pa t_e} &= \sum_{z_i\in\Ga\bigcap\N_h} \jm{\ga_ev_h}_{z_i} \frac{\pa^2 u}{\pa t_e \pa n_e}(z_i) -  \sum_{e\in\E_h^B}\ga_e\int_e\frac{\pa^3 u}{\pa t_e^2 \pa n_e}v_h, \label{part20_I2}\\
&= I_{2,1} + I_{2,2} + I_{2,3},\nn }
where
\eqn{ I_{2,1} &= \sum_{z_i\in\Ga\bigcap\N_h} \jm{\ga_e}_{z_i} \set{v_h}_{z_i} \frac{\pa^2 u}{\pa t_e \pa n_e}(z_i),\\
I_{2,2} &= \sum_{z_i\in\Ga\bigcap\N_h} \set{\ga_e}_{z_i}\jm{v_h}_{z_i} \frac{\pa^2 u}{\pa t_e \pa n_e}(z_i),\\
I_{2,3} &= - \sum_{e\in\E_h^B}\ga_e\int_e\frac{\pa^3 u}{\pa t_e^2 \pa n_e}v_h. }
Suppose that $w\in H^1([a,b])$ and denote by $h_{ab}=b-a$, then we have
\eq{
w^2(b) &= \int_a^b \bigg( \frac{x-a}{b-a} w^2(x) \bigg)' \dx = \frac{1}{b-a} \int_a^b w^2 + 2\int_a^b \frac{x-a}{b-a}ww' \label{eq_p2e}\\
& \leq \frac{1}{h_{ab}} \norm{w}_{L^2([a,b])}^2 + 2\abs{w}_{H^1([a,b])} \norm{w}_{L^2([a,b])}, \nn
}
which implies
\eq{
\abs{I_{2,1}} &\leq  \sum_{z_i\in\Ga\bigcap\N_h} \abs{\jm{\ga_e}_{z_i}} \bigg( \frac{1}{h_{e_i}} \abs{\frac{\pa u}{\pa n_{e_i}}}_{H^1(e_i)}^2 + 2\abs{\frac{\pa u}{\pa n_{e_i}}}_{H^2(e_i)}\abs{\frac{\pa u}{\pa n_{e_i}}}_{H^1(e_i)}\bigg)^{1/2} \label{part21_I21}\\
&\qquad \cdot \bigg( \frac{1}{h_{e_i}} \norm{v_h}_{L^2(e_i\cap e_{i+1})}^2 + 2\abs{v_h}_{\cH^1(e_i)}\norm{v_h}_{L^2(e_i\cap e_{i+1})} \bigg)^{1/2}\nn\\
&\ls \max_{z_i\in\Ga\bigcap\N_h} \abs{\jm{\ga_e}_{z_i}} \bigg( \frac{1}{h} \abs{\frac{\pa u}{\pa n}}_{H^1(\Ga)}^2 + \abs{\frac{\pa u}{\pa n}}_{H^2(\Ga)}\abs{\frac{\pa u}{\pa n}}_{H^1(\Ga)}\bigg)^{1/2}\nn\\
&\qquad \cdot \bigg( \frac{1}{h} \norm{v_h}_{L^2(\Ga)}^2 + \abs{v_h}_{\cH^1(\Ga)}\norm{v_h}_{L^2(\Ga)} \bigg)^{1/2}\nn\\
&\ls \frac{\max_{z_i\in\Ga\bigcap\N_h}\abs{\jm{\ga_e}_{z_i}}}{h} \bigg( \big(\abs{g}_{H^1(\Ga)}^2 + k\abs{u}_{H^1(\Ga)}^2\big) + h\big(\abs{g}_{H^2(\Ga)}+k\abs{u}_{H^2(\Ga)}\big) \cdot \nn\\
&\quad\big(\abs{g}_{H^1(\Ga)}+k\abs{u}_{H^1(\Ga)}\big) \bigg)^{1/2} \cdot \bigg( \norm{v_h}_{L^2(\Ga)}^2 + h \abs{v_h}_{\cH^1(\Ga)}\norm{v_h}_{L^2(\Ga)} \bigg)^{1/2}\nn\\
&\ls h^{1+\al} k^{3/2} \bigg( \norm{v_h}_0 \norm{v_h}_{\cH^1} + h^{1/2} \norm{v_h}_{\cH^1}^{3/2}\norm{v_h}_0^{1/2}\bigg)^{1/2} C_{u,g}\nn\\
&\ls kh^{1+\al} \norme{v_h}_{1,h} C_{u,g},\nn
}
where we have used \eqref{beta_gamma4}.

For any $z_i\in\N_h^B$, let $\tau_{i,1},\ \tau_{i,2},\ \cdots,\ \tau_{i,n_i}$ be $n_i$ triangles clockwise around $z_i$ (cf. \ref{figTri}).
If $n_i=1$ for some $z_i\in\N_h^B$, it is easy to see that $\jm{\ga_ev_h}_{z_i}$ in \eqref{part20_I2} is equal to zero. Thus we assume that $n_i\geq2$ for all $z_i\in\N_h^B$ for simplicity of presentation. Denote by $e_{ij}=\tau_{i,j}\cap \tau_{i,j+1}$ for $j=1,2,\cdots,n_i-1$. Clearly, $e_{ij}$ are in $\E_h^I$ and we have
\eqn{ \abs{\jm{v_h}_{z_i}} \leq \sum_{j=1}^{n_i-1} \abs{\jm{v_h}_{e_{ij}}} \leq \sqrt{n_i-1} \bigg( \sum_{j=1}^{n_i-1} \abs{\jm{v_h}_{e_{ij}}}^2 \bigg)^{1/2}. }
It is well known that $n_i$ can be bounded by a constant independent of the mesh size $h$ for any regular triangulation. Therefore, we have

\eq{
\abs{I_{2,2}} &\leq  \sum_{z_i\in\N_h^B} \abs{\set{\ga_e}_{z_i}} \bigg( \frac{1}{h_{e_i}} \abs{\frac{\pa u}{\pa n_{e_i}}}_{H^1(e_i)}^2 + 2\abs{\frac{\pa u}{\pa n_{e_i}}}_{H^2(e_i)}\abs{\frac{\pa u}{\pa n_{e_i}}}_{H^1(e_i)}\bigg)^{1/2} \label{part21_I22} \\
&\qquad \cdot \bigg( \sum_{j=1}^{n_i-1} \frac{1}{h_{e_{ij}}} \norm{\jm{v_h}}_{L^2(e_{ij})}^2 + 2\abs{\jm{v_h}}_{H^1(e_{ij})}\norm{\jm{v_h}}_{L^2(e_{ij})} \bigg)^{1/2}\nn\\
&\ls \max_{z_i\in\N_h^B} \abs{\set{\ga_e}_{z_i}} \bigg( \frac{1}{h} \abs{\frac{\pa u}{\pa n}}_{H^1(\Ga)}^2 + \abs{\frac{\pa u}{\pa n}}_{H^2(\Ga)}\abs{\frac{\pa u}{\pa n}}_{H^1(\Ga)}\bigg)^{1/2}\nn\\
&\qquad \cdot \bigg( \frac{h^\mu}{\rho_{0}} \sum_{z_i\in\N_h^B} \sum_{j=1}^{n_i-1}  \frac{\rho_{0,e_{ij}}}{h_{e_{ij}}^{1+\mu}} \norm{\jm{v_h}}_{L^2(e_{ij})}^2  \bigg)^{1/2} \nn\\
&\ls h^{(\mu-1)/2}\max_{z_i\in\N_h^B} \abs{\set{\ga_e}_{z_i}} \bigg( \big(\abs{g}_{H^1(\Ga)}^2 + k\abs{u}_{H^1(\Ga)}^2\big) + h\big(\abs{g}_{H^2(\Ga)} \nn\\
&\qquad +k\abs{u}_{H^2(\Ga)}\big) \cdot \big(\abs{g}_{H^1(\Ga)}+k\abs{u}_{H^1(\Ga)}\big) \bigg)^{1/2} J_0(v_h,v_h)^{1/2} \nn\\
&\ls k^{3/2}h^{(3+\mu)/2} \norme{v_h}_{1,h} C_{u,g}\nn
}
where we have used \eqref{beta_gamma3}.

On the other hand,
\eq{
\abs{I_{2,3}} &\ls h^2 \abs{\frac{\pa u}{\pa n}}_{H^2(\Ga)} \norm{v_h}_{L^2(\Ga)}\label{part22_I2}\\
& \ls h^2 \big(  \abs{g}_{H^2(\Ga)} + k \abs{u}_{H^2(\Ga)} \big) \cdot \norm{v_h}_0^{1/2} \norm{v_h}_{\cH^1}^{1/2} \nn\\
& \ls k^{3/2} h^2  \norme{v_h}_{1,h} C_{u,g}. \nn
}
From \eqref{part20_I2}--\eqref{part22_I2},
\eq{ \abs{\sum_{e\in\E_h^B}\ga_e\int_e\frac{\pa^2 u}{\pa t_e \pa n_e}\frac{\pa v_h}{\pa t_e}} \ls \big((kh)^2+kh^{1+\al}+k^{3/2}h^{(3+\mu)/2}\big)\norme{v_h}C_{u,g}. \label{part2_I2} }
Then the estimate of $I_2$ can be obtained from \eqref{part1_I2} and \eqref{part2_I2},
\eq{ \abs{I_2} \ls \big((kh)^2+kh^{1+\al}+k^{3/2}h^{(3+\mu)/2}\big)\norme{v_h}_{1,h}C_{u,g}. \label{I2est} }

Finally, by combining \eqref{I1est} and \eqref{I2est}, we prove the lemma.
\end{proof}

\begin{theorem} \label{uhui}
Assume that $\T_h$ satisfies {the} \emph{mesh condition $\al$}. $u_h$ is the DG finite element solution of the scheme~\eqref{eq_DG-FEM} and $u_I$ is the linear interpolation of the solution $u$ to \eqref{eq1.1a}--\eqref{eq1.1b}. There exists a constant $C_0$
independent of $k$ and $h$, such that if $k(kh)^2\leq C_0$, we have
\eq{ \norme{u_h-u_I}_{1,h} \ls \left(kh^{1+\al}+kh^{1+\mu/2}+k(kh)^2\right) C_{u,g}. \label{eq:th1eq1} }
\end{theorem}

\begin{proof}
For simplicity of presentation, { we denote} $v_h=u_h-u_I$. By Lemma~\ref{lemma_norms_equal} and the Galerkin orthogonality, we have
\eq{
&\norme{u_h-u_I}_{1,h}^2 \ls  a_h(u_h-u_I,v_h) + k^2(u_h-u_I,v_h)  \label{eq1}\\
&= \Re\big( a_h(u_h-u_I,v_h) - k^2(u_h-u_I,v_h) + \i k\pd{u_h-u_I,v_h} + 2k^2(u_h-u_I,v_h)\big)  \nn\\
&= \Re\big( a_h(u-u_I,v_h) - k^2(u-u_I,v_h) + \i k\pd{u-u_I,v_h} + 2k^2(u_h-u_I,v_h)\big) \nn\\
&= I_1 + I_2 + I_3 + I_4.\nn
}
{ It is} well known that
\eq{ \abs{I_2} \leq k\norm{u-u_I}_{0} k\norm{v_h} \ls kh^2\norm{u}_{2} \norme{v_h}_{1,h}. \label{eq_I2}}
By the trace inequality, we have
\eq{
\abs{I_3} &\leq k\norm{u-u_I}_{L^2(\pa\Om)}\norm{v_h}_{L^2(\pa\Om)} \label{eq_I3}   \\
&\ls  kh^2\norm{u}_{H^2(\pa\Om)}\norm{v_h}_{L^2(\pa\Om)}   \nn\\
&\ls k^{1/2}h^2\norm{u}^{1/2}_{2}\norm{u}^{1/2}_{3}\norme{v_h}_{1,h}    \nn\\
&\ls (kh)^2 \norme{v_h}_{1,h} C_{u,g}.\nn
}
From Lemma~\ref{lemma_pre_err}, we know that there exists a constant $C_0$ independent on $k$ and $h$, such that if $k(kh)^2\leq C_0$,
the following inequality holds,
\eqn{ k\norm{u-u_h}_{L^2(\Om)} \ls \left((kh)^2+k(kh)^2\right) C_{u,g}, }
which implies
\eq{ \abs{I_4} &\leq 2k\norm{u_h-u_I}_{0} \cdot k\norm{v_h}_{0} \label{eq_I4}\\
&\leq 2\left(k\norm{u-u_h}_{0}+k\norm{u-u_I}_{0}\right)\cdot k\norm{v_h}_{0}   \nn\\
&\ls  \left((kh)^2+k(kh)^2\right) \norme{v_h}_{1,h} C_{u,g}.\nn
}
Next, we estimate $I_1$. From the definition of $a_h(\cdot,\cdot)$ and the fact that $u_I$ is continuous in $\Om$, we have
\eqn{ \abs{I_1} &\leq  \abs{ \sum_{\tau\in\T_h} \int_\tau\na(u-u_I)\cdot\na\bar v_h } + \abs{ \sum_{e\in\E_h^I} \pd{\set{\frac{\pa (u-u_I)}{\pa \bn_e}},\jm{v_h}}_e} \\
&= I_{1,1} + I_{1,2}.   }
Then
\eq{ I_{1,2} &\leq \sum_{e\in\E_h^I} \norm{\set{\frac{\pa (u-u_I)}{\pa \bn_e}}}_{L^2(e)} \norm{\jm{v_h}}_{L^2(e)} \label{eq_I12}\\
&\ls h^{1/2} \sum_{e\in\E_h^I} \norm{u}_{H^2(\tau_e)} \norm{\jm{v_h}}_{L^2(e)} \nn\\
&\ls h^{1+\mu/2} \norm{u}_2 J_0(v_h,v_h)^{1/2} \nn\\
&\ls kh^{1+\mu/2} \norme{v_h}_{1,h} C_{u,g}. \nn }
From Lemma~\ref{uui} and \eqref{eq_I12},
\eq{ \abs{I_1} \ls \left((kh)^2+kh^{1+\al}+kh^{1+\mu/2}\right)\norme{v_h}_{1,h} C_{u,g}. \label{eq_I1}  }
Therefore, by combining the equations \eqref{eq1}--\eqref{eq_I4} and \eqref{eq_I1}, we have if $k(kh)^2\leq C_0$, the following estimate holds:
\eqn{ \norme{u_h-u_I}_{1,h}^2 \ls  \left((kh)^2+kh^{1+\al}+kh^{1+\mu/2}\right)\norme{v_h}_{1,h} C_{u,g}. }
This completes the proof.
\end{proof}

In this paper, for any node point $z\in\N_h$,  let $n_z$ be the number of triangles associated with $z$ and let $\tau_{z,1},\ \tau_{z,2},\cdots,\ \tau_{z,n_z}$ be the elements having $z$ such that they are counterclockwise around $z$ as shown in Figure~\ref{figTri}. $\tau_{z,j}$ means $\tau_{z,\mathrm{mod}(j,n_z)}$ for the integer $j>n_z$.
\begin{figure}
\begin{center}
\setlength{\unitlength}{0.25mm}
  \begin{picture}(390,110)

  \linethickness{0.25mm}
  \put(195,55){\line(1,0){100}}
  \put(195,55){\line(-1,0){100}}
  \put(195,55){\line(0,1){50}}
  \put(195,55){\line(0,-1){50}}
  \put(95,55){\line(2,1){100}}
  \put(195,55){\line(2,1){100}}
  \put(95,5){\line(2,1){100}}
  \put(195,5){\line(2,1){100}}
  \put(95,5){\line(1,0){100}}
  \put(95,5){\line(0,1){50}}
  \put(195,105){\line(1,0){100}}
  \put(295,55){\line(0,1){50}}

  \put(195,45){$z$}
  \put(205,25){$\tau_{z,1}$}
  \put(245,65){$\tau_{z,2}$}
  \put(205,85){$\tau_{z,3}$}
  \put(165,75){$\tau_{z,4}$}
  \put(105,35){$\tau_{z,5}$}
  \put(165,25){$\tau_{z,6}$}

  \end{picture}
\end{center}
\caption{$n_z$ triangles sharing $z$ with $n_z=6$.}
\label{figTri}
\end{figure}
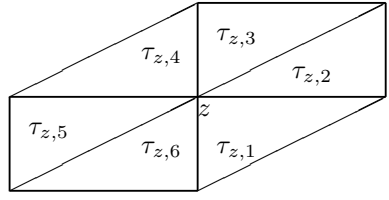

\begin{lemma}
Assume that $\T_h$ satisfies {the} \emph{mesh condition $\al$} and $u_h$ is the DG finite element solution of the scheme~\eqref{eq_DG-FEM}. For any $z\in\N_h$ and $\tau\in\T_h$ having the node point $z$, let $u_h(z,\tau)$ be $u_h|_\tau(z)$.   There exists a constant $C_0$ independent of $k$ and $h$, such that if $k(kh)^2\leq C_0$,
we have
\eqn{ \left( \sum_{z\in\N_h}\sum_{i=1}^{n_z}\abs{u_h(z,\tau_{i+1})-u_h(z,\tau_i)}^2 \right)^{1/2}  \ls h^{\mu/2} \left(kh^{1+\al}+kh^{1+\mu/2}+k(kh)^2\right) C_{u,g}.  }
\end{lemma}
\begin{proof}
For any $e\in\E_h^I$, let $z_e,\ z_e'$ be its two endpoints and $\tau_e,\ \tau_e'$ be two elements sharing $e$ . Then from \eqref{eq_p2e} and the discrete inverse inequalities, we have
\eqn{ &\bigg( \sum_{z\in\N_h}\sum_{i=1}^{n_z}\abs{u_h(z,\tau_{i+1})-u_h(z,\tau_i)}^2 \bigg)^{1/2} \\
&= \bigg( \sum_{e\in\E_h^I}\abs{u_h(z_e,\tau_e)-u_h(z_e,\tau_e')}^2 + \abs{u_h(z_e',\tau_e)-u_h(z_e',\tau_e')}^2 \bigg)^{1/2} \\
&\ls \bigg( \sum_{e\in\E_h^I} \frac{1}{h_e} \norm{\jm{u_h}}_{L^2(e)}^2 + \norm{\jm{u_h}}_{L^2(e)} \norm{\jm{u_h}}_{H^1(e)} \bigg)^{1/2} \\
&\ls \bigg( \sum_{e\in\E_h^I} \frac{1}{h_e} \norm{\jm{u_h}}_{L^2(e)}^2 \bigg)^{1/2}.  }
 From the fact that the linear interpolation $u_I$ of the solution $u$ to \eqref{eq1.1a}--\eqref{eq1.1b} is continuous , that is $\jm{u_I}_e=0\ \forall e\in\E_H^I$, and Theorem~\ref{uhui},
\eqn{ & \bigg( \sum_{e\in\E_h^I} \frac{1}{h_e} \norm{\jm{u_h}}_{L^2(e)}^2  \bigg)^{1/2} = \bigg( \sum_{e\in\E_h^I} \frac{1}{h_e} \norm{\jm{u_h-u_I}}_{L^2(e)}^2 \bigg)^{1/2}\\
&\ls h^{\mu/2} J_0(u_h-u_I,u_h-u_I)^{1/2} \ls h^{\mu/2} \norme{u_h-u_I}_{1,h} \\
&\ls h^{\mu/2} \left((kh)^2+kh^{1+\al}+kh^{1+\mu/2}\right) C_{u,g}. }
This completes the proof.
\end{proof}

\section{The gradient recovery operator $G_h$ and its superconvergence}\label{sec_ppr}
In this section, we define a polynomial preserving recovery method for the discontinuous finite element space and derive the superconvergent error estimate.

We first recall a gradient recovery operator developed in 2004 for the continuous finite element methods, which is called polynomial
preserving recovery (PPR). Let $\tilde{V}_h$ be the approximation space of continuous piecewise linear polynomials over $\T_h$ and let $\tilde{G}_h:C(\Om)\mapsto \tilde{V}_h\times \tilde{V}_h$ be the gradient recovery operator. Given a node $z\in\N_h$, we select $n\geq6$ sampling points $z_j\in\N_h$, $j=1,2,\cdots,n$, in an element
patch $\om_z$ containing $z$ ($z$ is one of $z_j$) and fit a polynomial of degree $2$, in the least
squares sense, with values of $w\in C(\Om)$ at those sampling points. First, we find $p_2\in P_2(\om_z)$ for some
$w\in C(\Om)$ such that
\eq{ \sum_{j=1}^n(p_2-w)^2(z_j)=\min_{q\in P_2}\sum_{j=1}^n(q-w)^2(z_j). }
Here $P_2(\om_z)$ is the well-known piecewise quadratic polynomial space defined on $\om_z$.
The recovery gradient at $z$ is then defined as
\eq{ \tilde G_hw(z)=(\na p_2)(z). }

We define a PPR operator $G_h$ over the DG finite element space.  For any $u_h\in V_h$, we define a continuous piecewise linear polynomial $\tilde{u}_h\in\tilde{V}_h$ by
\eqn{ \tilde{u}_h(z) = \la_1 u_h(z,\tau_{z,1}) + \la_2 u_h(z,\tau_{z,2}) + \cdots + \la_{n_z} u_h(z,\tau_{z,n_z})\quad \forall z\in\N_h,   }
where $\la_j,\ j=1,2,\cdots,n_z$,are non-negative numbers satisfying $\la_1+\la_2+\cdots+\la_{n_z}=1$.

We define the gradient recovery operator $G_h$ from $V_h$ to $\tilde{V}_h\times \tilde{V}_h$ by
\eq{ G_hu_h := \tilde{G}_h\tilde{u}_h. }
Clearly, for any $v_h\in\tilde{V}_h$, $G_hv_h$ is completely equal to $\tilde{G}_hv_h$.

\begin{lemma}
For any node point $z\in\N_h$,  let $\tau_{z,1},\ \tau_{z,2},\cdots,\ \tau_{z,n_z}$ be the $n_z$ elements counterclockwise around $z$ as shown in Figure~\ref{figTri}.
The following inequality holds for any $u_h\in V_h$
\eq{ \norm{G_hu_h}_0 \ls \abs{u_h}_{\cH^1(\T_h)} + \bigg( \sum_{z\in\N_h} \sum_{j=1}^{n_z-1} |u_h(z,\tau_{z,j+1})-u_h(z,\tau_{z,j})|^2 \bigg)^{1/2}. }
\end{lemma}
\begin{proof}
We have the property $\norm{\tilde G_h\tilde v_h}_0\ls \norm{\na\tilde v_h}_0$ for $\tilde v_h\in \tilde V_h$ (cf. \cite{nz04}), which implies
\eqn{ \norm{G_hu_h}_0 = \norm{\tilde{G}_h\tuh}_0 \ls \norm{\na\tuh}_0 \ls \abs{\tuh-u_h}_{\cH^1(\T_h)} + \abs{u_h}_{\cH^1(\T_h)}.}
For any $\tau\in\T_h$, let $\phi_{\tau,1},\ \phi_{\tau,2}$ and $\phi_{\tau,3}$ be its node bases and let $z_{\tau,1},\ z_{\tau,2},\ z_{\tau,3}\in\N_h$ be its three vertices satisfying $\phi_{\tau,i}(z_{\tau,j})=\delta(i-j)$. We have
\eqn{ &\abs{\tuh-u_h}_{\cH^1(\T_h)} = \bigg( \sum_{\tau\in\T_h} \norm{\na(\tuh-u_h)}_{L^2(\tau)}^2 \bigg)^{1/2} \\
&= \bigg( \sum_{\tau\in\T_h}  \int_\tau \big| \sum_{j=1}^3 (\tuh(z_{\tau,j})-u_h(z_{\tau,j},\tau))\pa_x\phi_{\tau,j} \big|^2  \\
&\qquad +  \int_\tau \big| \sum_{j=1}^3 (\tuh(z_{\tau,j})-u_h(z_{\tau,j},\tau))\pa_y\phi_{\tau,j} \big|^2 \bigg)^{1/2}  \\
&\leq \bigg( \sum_{\tau\in\T_h}  \big(\sum_{j=1}^3 |\tuh(z_{\tau,j})-u_h(z_{\tau,j},\tau)|^2\big) \cdot \big(\sum_{j=1}^3\int_\tau|\pa_x\phi_{\tau,j}|^2+|\pa_y\phi_{\tau,j}|^2\big)  \bigg)^{1/2}.  }

Since $\T_h$ is a uniform regular triangulation, it is well known that $\sum_{j=1}^3\int_\tau|\pa_x\phi_{\tau,j}|^2+|\pa_y\phi_{\tau,j}|^2$ can be bounded by some constant $C$ independent of the mesh size $h$ and the triangle $\tau$ for any $\tau\in\T_h$. Therefore, we have

\eqn{ \abs{\tuh-u_h}_{\cH^1(\T_h)} &\ls \bigg( \sum_{\tau\in\T_h}  \big(\sum_{j=1}^3 |\tuh(z_{\tau,j})-u_h(z_{\tau,j},\tau)|^2\big)  \bigg)^{1/2}  \\
&\ls \bigg( \sum_{z\in\N_h} \sum_{j=1}^{n_z} |\tuh(z)-u_h(z,\tau_{z,j})|^2 \bigg)^{1/2}.  }
From the definition of $\tuh$, we have for any $z\in\N_h$,
\eqn{ &\sum_{j=1}^{n_z} |\tuh(z)-u_h(z,\tau_{z,j})|^2 = \sum_{j=1}^{n_z} | \sum_{i\neq j} \la_i(u_h(z,\tau_{z,i})-u_h(z,\tau_{z,j})) |^2 \\
&\leq \sum_{j=1}^{n_z} \big( \sum_{i\neq j} \la_i\abs{u_h(z,\tau_{z,i})-u_h(z,\tau_{z,j})} \big)^2 \\
&\leq \sum_{j=1}^{n_z} (n_z-1)\sum_{i\neq j} \la_i^2\abs{u_h(z,\tau_{z,i})-u_h(z,\tau_{z,j})}^2  \\
&\leq (n_z-1) \sum_{j=1}^{n_z-1} \sum_{i=j+1}^{n_z} (\la_i^2+\la_j^2)\abs{u_h(z,\tau_{z,i})-u_h(z,\tau_{z,j})}^2  \\
&\leq (n_z-1) \sum_{j=1}^{n_z-1} \sum_{i=j+1}^{n_z} (\la_i^2+\la_j^2) \big(\sum_{t=j}^{i-1}\abs{u_h(z,\tau_{z,t+1})-u_h(z,\tau_{z,t})}\big)^2 \\
&\leq (n_z-1)^2 \sum_{j=1}^{n_z-1} \sum_{i=j+1}^{n_z} \sum_{t=i}^{j-1} (\la_i^2+\la_j^2) \abs{u_h(z,\tau_{z,t+1})-u_h(z,\tau_{z,t})}^2 \\
&\leq (n_z-1)^2 \sum_{t=1}^{n_z-1} \big((n_z-t)\sum_{i=1}^t \la_i^2+t\sum_{i=t+1}^{n_z}\la_i^2\big) \abs{u_h(z,\tau_{z,t+1})-u_h(z,\tau_{z,t})}^2 \\
&\leq (n_z-1)^3 \sum_{t=1}^{n_z-1} \abs{u_h(z,\tau_{z,t+1})-u_h(z,\tau_{z,t})}^2. }
Since $n_z$ can be bounded by a constant independent of $h$ and the vertex $z$, we complete the proof.
\end{proof}

\begin{lemma}\label{ghphi}
For any element $\tau\in\T_h$ and any function $\phi\in H^3(\tilde\tau)$,
\eq{ \norm{G_h\phi_I-\na \phi}_{L^2(\tau)}\ls h^2\norm{\phi}_{H^3(\tilde\tau)}, }
where $\tilde\tau=\bigcup\set{\om_z:z\in\N_h\cap\tau}$ and $\phi_I$ is the linear interpolant of $\phi$.
\end{lemma}
\begin{proof}
The proof is completed by the fact that $G_hu_I=\tilde G_hu_I$ and Lemma~4.1 in ??.
\end{proof}

Since $u_I$ is continuous, that is
\eqn{ \tilde{u}_I(z) - u_I(z,\tau_{z,j}) = 0\quad \forall z\in\N_h\ \mathrm{and}\ j=1,2,\cdots,n_z, }
we have
\eq{ &\norm{G_hu_h-\na u}_0 \leq \norm{G_h(u_h-u_I)}_0 + \norm{G_hu_I-\na u}_0 \label{eq_th_a}\\
&\ls \abs{u_h-u_I}_{\cH^1(\T_h)} + \norm{G_hu_I-\na u}_0 + \bigg( \sum_{z\in\N_h} \sum_{j=1}^{n_z} |u_h(z)-u_h(z,\tau_{z,j})|^2 \bigg)^{1/2}. \nn }
Then by combining Lemmas~\ref{uui}--\ref{ghphi} and the inequality \eqref{eq_th_a}, we have the following theorem which is our main result in the paper.
\begin{theorem} \label{main1}
Let $u$ and $u_h$ be the solutions to \eqref{eq1.1a}--\eqref{eq1.1b} and the discrete solution, respectively.
Assume that $\T_h$ satisfies the \emph{mesh condition $\al$}.
Then there exists a constant $C_0$ independent of $k$ and $h$ such that if $k(kh)^2\leq C_0$,
\eq{ \norm{G_hu_h-\na u}_0\ls \left(kh^{1+\al}+kh^{1+\mu/2}+k(kh)^2\right) C_{u,g}. \label{ghpoll} }
\end{theorem}

\section{The influence of the operator $G_h$ to the pollution error} \label{sec_ppr_nabla}
In this section, we estimate the error between $G_hu_h$ and $\na u_h$, which motivate us to combine the Richardson extrapolation and the ppr technique to reduce the numerical errors, and define the a posterior estimator in Section~\ref{sec_num}. 

First, we define an elliptic projection from $V$ to $V_h$:
find $u_h^+\in V_h$ such that
\eq{ a_h(u_h^+,v_h)+\i k\pd{u_h^+,v_h}=a_h(u,v_h)+\i k\pd{u,v_h}\quad \forall v_h\in V_h.\label{ellpro}}
In other words, the elliptic projection $u_h^+$ of $u$ is the finite element approximation
to the solution of the following (complex-valued) Poisson problem:
\eq{ -\De u &= F \quad \mathrm{in} \quad \Om,\label{eq2.2a}\\
\frac{\pa u}{\pa n} + \i ku &= g \quad \mathrm{on}\quad \Ga,\label{eq2.2b}  }
for some given function $F$ which are determined by $u$.
This kind of elliptic projection is often used to study some properties,
such as stability and convergence, of the FEM for the Helmholtz problem.  Readers are referred to \cite{zw,zd,dzh,dw}.

\begin{lemma} \label{ellemm}
Assume that $u$ is $H^2$-regular. $u_h^+$ is its elliptic projection defined by \eqref{ellpro}.
There hold the following estimates:
\eqn{ &\norm{u-u_h^+}_{1,h}\ls \inf_{v_h\in V_h} \norme{u-v_h}_{1,h} \\
&\norm{u-u_h^+}_0\ls h \inf_{v_h\in V_h} \norme{u-v_h}_{1,h}.}
\end{lemma}
\begin{proof}
From Lemma~3.5 in \cite{zd}, we know that
\eqn{
\norm{u-u_h^+}_{1,h} &\ls \inf_{v_h\in V_h} \big(\norm{u-v_h}_1^2+k\norm{u-v_h}_{L^2(\Ga)}^2\big)^{1/2},\\
\norm{u-u_h^+}_0 &\ls h \inf_{v_h\in V_h} \big(\norm{u-v_h}_1^2+k\norm{u-v_h}_{L^2(\Ga)}^2\big)^{1/2}.
}
Then the estimates follow from
\eqn{ k\norm{u-v_h}_{L^2(\Ga)}^2 &\ls k\norm{u-v_h}_0\norm{u-v_h}_{\cH^1(\Om)}\\
&\ls k^2\norm{u-v_h}_0^2 + \norm{u-v_h}_{\cH^1(\Om)}^2 \ls \norme{u-v_h}_{1,h}^2.
}
\end{proof}

\begin{lemma}\label{naphu}
Assume that $u$ is the exact solution to \eqref{eq1.1a}--\eqref{eq1.1b} and $u_h^+$ is its elliptic projection defined by \eqref{ellpro}.
Assume that $\T_h$ satisfies the \emph{mesh condition $\al$}. We have
\eq{
\norme{\na u_h^+-\na u_I}_{1,h} \ls \big(kh^{1+\al}+kh^{1+\mu/2}+(kh)^2\big) C_{u,g}. \label{eq:nabla-phu-ui}
}
\end{lemma}
\begin{proof}
Denote $v_h=u_h^+-u_I$. By the Galerkin orthogonality,
\eqn{ &\norme{\na u_h^+-\na u_I}_{1,h}^2 \ls\Re\big( a_h(u_h^+-u_I,v_h) + \i k\pd{u_h^+-u_I,v_h} \big)\\
&\qquad + k^2(u_h^+-u_I,v_h)\\
&\ls\Re\big( a_h(u-u_I,v_h) + \i k\pd{u-u_I,v_h} \big)  + k^2(u_h^+-u_I,v_h)\\
&\ls \abs{a_h(u-u_I,v_h)} + \abs{k\pd{u-u_I,v_h}} + k\norm{u_h^+-u_I}_0\cdot k\norm{v_h}_0. }
By some arguments same to those in \ref{uhui}, it is obtained that
\eqn{ \abs{a_h(u-u_I,v_h)} \ls \left((kh)^2+kh^{1+\al}+kh^{1+\mu/2}\right)\norme{v_h}_{1,h} C_{u,g}}
and
\eqn{\abs{k\pd{u-u_I,v_h}}\ls (kh)^2 \norme{v_h}_{1,h} C_{u,g}.}
From Lemma~\ref{ellemm}, we have
\eqn{ k\norm{u_h^+-u_I}_0\cdot k\norm{v_h}_0 &\leq \big(k\norm{u_h^+-u}_0+k\norm{u-u_I}_0\big)\norme{v_h}_{1,h}\\
&\ls (kh)^2\big(1+\sqrt{kh}\big)\norme{v_h}_{1,h} C_{u,g}. }
By combining the inequalities above, we complete the proof.
\end{proof}

\begin{theorem} \label{main2}
Let $u$ and $u_h^+$ be the solution to \eqref{eq1.1a}--\eqref{eq1.1b} and the elliptic projection defined by \eqref{ellpro}, respectively. Assume that the \emph{mesh condition $\al$} is satisfied.
Then the following error estimate holds:
\eq{ \norm{G_hu_h^+-\na u}_0\ls \big(kh^{1+\al}+kh^{1+\mu/2}+(kh)^2\big)  C_{u,g}. \label{phusup} }
\end{theorem}

\begin{theorem}\label{main3}
Let $u$ and $u_h^+$ be the solution to \eqref{eq1.1a}--\eqref{eq1.1b} and the elliptic projection defined by \eqref{ellpro}, respectively. Assume that the \emph{mesh condition $\al$} is satisfied. We have
\eq{ \bigg(\sum_{\tau\in\T_h}\norm{G_hu_h-\na u_h}_{L^2(\tau)}^2\bigg)^{1/2}\ls \big(kh+k(kh)^3\big) C_{u,g}. \label{eq:th2eq1} }
\end{theorem}
\begin{proof}
Denote by $\ta_h=u_h-u_h^+$. $u_h$ can be written as $u_h=u_h^++\ta_h$, where $u_h^+$ is the elliptic projection of $u$ defined by \eqref{ellpro}. We have
\eq{ &\bigg(\sum_{\tau\in\T_h}\norm{G_hu_h-\na u_h}_{L^2(\tau)}^2\bigg)^{1/2} \label{main3eq1}\\
&= \bigg(\sum_{\tau\in\T_h}\norm{G_h(u_h^++\ta_h)-\na (u_h^++\ta_h)}_{L^2(\tau)}^2\bigg)^{1/2}\nn\\
&\leq \bigg(\sum_{\tau\in\T_h}\norm{G_hu_h^+-\na u_h^+}_{L^2(\tau)}^2\bigg)^{1/2} + \bigg(\sum_{\tau\in\T_h}\norm{G_h\ta_h-\na \ta_h}_{L^2(\tau)}^2\bigg)^{1/2}. \nn}
From Lemma~\ref{ellemm} and Theorem~\ref{main2},
\eq{ &\bigg(\sum_{\tau\in\T_h}\norm{G_hu_h^+-\na u_h^+}_{L^2(\tau)}^2\bigg)^{1/2} \label{main3eq2}\\
&\leq \norm{G_hu_h^+-\na u}_0 + \abs{u-u_h^+}_{\cH^1}\nn\\
&\ls \big(kh^{1+\al}+kh^{1+\mu/2}+(kh)^2\big)C_{u,g} + khC_{u,g} \nn\\
&\ls (kh+(kh)^2)C_{u,g}. \nn }
From \eqref{eq_DG-FEM} and \eqref{ellpro}, we have that $\ta_h$ satisfies the following equation:
\eq{ a_h(\ta_h,v_h) + \i k\pd{\ta_h,v_h} = - k^2(u-u_h,v_h). }
Therefor, $\ta_h$ can be understood as the numerical approximation to the following Poisson problem
with Robin boundary:
\eqn{ -\De \ta &= -k^2(u-u_h) \quad\mathrm{in} \quad \Om,\\
\frac{\pa \ta}{\pa n} + \i k\ta &= 0\quad\quad\quad\quad\qquad\mathrm{on}\quad \Ga.  }
Therefore,
\eq{ \bigg(\sum_{\tau\in\T_h}\norm{G_h\ta_h-\na \ta_h}_{L^2(\tau)}^2\bigg)^{1/2} &\leq \norm{G_h\ta_h-\na \ta}_0 +  \abs{\ta-\ta_h}_{\cH^1}\label{main3eq3}\\
&\ls h\norm{\ta}_2\ls k^2h\norm{u-u_h}_0\nn\\
&\ls k^2h(kh^2+k^2h^2)C_{u,g}\nn\\
&\ls ((kh)^3+k(kh)^3)C_{u,g}.\nn }
The proof is completed by combining \eqref{main3eq1}--\eqref{main3eq3}.
\end{proof}

\section{Numerical Tests}\label{sec_num}
In this section, some numerical tests are implemented in order to demonstrate our theoretical results. We simulate
the Helmholtz problem \eqref{eq1.1a}--\eqref{eq1.1b} where the source data $f$ and $g$ is so chosen that the
exact solution is
\eqn{ u=\frac{\cos(k r)}{r}-\frac{\cos k+\i\sin k}{k\big(J_0(k)+\i J_1(k)\big)}J_0(k r) }
in polar coordinates, where $J_0(z)$ is a Bessel function of the first kind, and $\Om=[0.5,1.5]\times[0.5,1.5]$.
In this section, let $\ga_0=5$ and we set $\la_1=1,\la_2=0,\cdots,\la_{n_z}=0$ in the definition of $G_h$.

We first simulate the problem over the regular pattern uniform triangulation and denote by $\T_N$ the triangulation consisting of $2N^2$ triangles of size $h$ which is equivalent to $1/N$.

From Theorem~\ref{uhui}, there exists the estimate that if $k(kh)^2\leq C_0$,
\eqn{ \norm{u_h-u_I}_{1,h} \ls kh^{1+\mu/2}+k(kh)^2. }
The left graphs in Figure~\ref{fig_error1}--Figure~\ref{fig_error3} show the numerical errors $\norm{u_h-u_I}_{1,h}$ with penalty parameters $\mu=0,1,2$ for $k=5,10,50$ and $100$, respectively. The right graphs in Figure~\ref{fig_error1}--Figure~\ref{fig_error3} show the convergence orders of the errors $\norm{u_h-u_I}_{1,h}$ shown in the left graphs, respectively. As we expected, $\norm{u_h-u_I}_{1,h}$ decays at the rates of $O(h^1)$, $O(h^{3/2})$ and $O(h^{2})$ for the small wave numbers $k=5,10$, respectively. However, we can see that for the large wave number $k=50,100$, $\norm{u_h-u_I}_{1,h}$ does not converge at first, then begins to decay at the rates which are greater than $O(h^{1+\mu/2})$ when $N$ is large enough, which implies the existence of the constraint $k(kh)^2\leq C_0$ and the so-called pollution error $k(kh)^2$.

Figure~\ref{fig_ppr1}--Figure~\ref{fig_ppr3} show the numerical errors $\norm{G_hu_h-\na u}_0$ in left graphs and the convergence order in right graphs for $k=5,10,50,100$ with penalty parameters $\mu=0,1,2$, respectively. Clearly, the recovered gradients super-converge at the rate greater than $O(h^{1+\mu/2})$. Therefore, whether the estimate~\eqref{ghpoll} is sharp with respect to $\mu$ is still open. The constraint $k(kh)^2\leq C_0$ and the so-called pollution error $k(kh)^2$ can also be observed.

\begin{figure}[htbp]
\begin{center}
\includegraphics[width=0.49\textwidth]{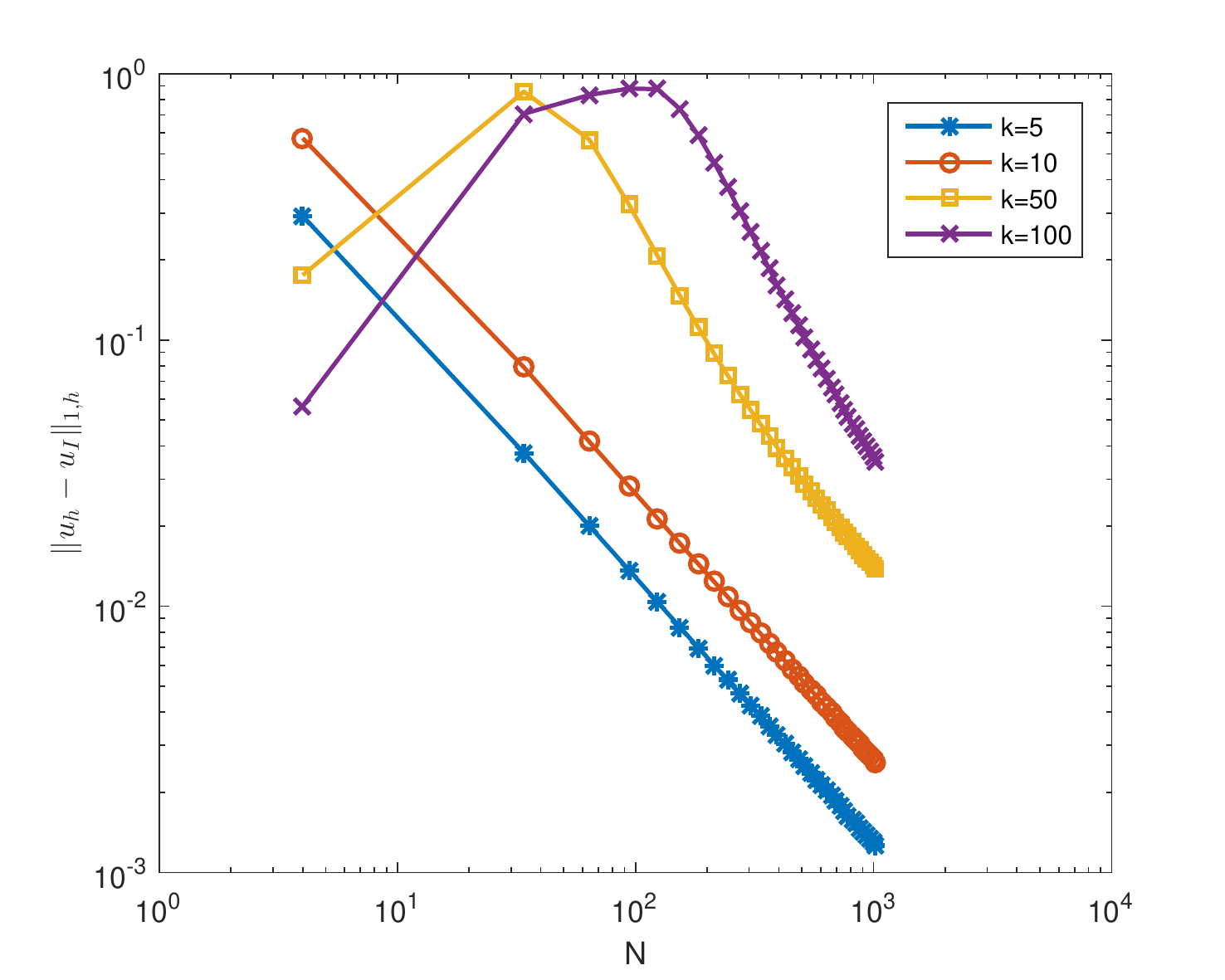}
\includegraphics[width=0.49\textwidth]{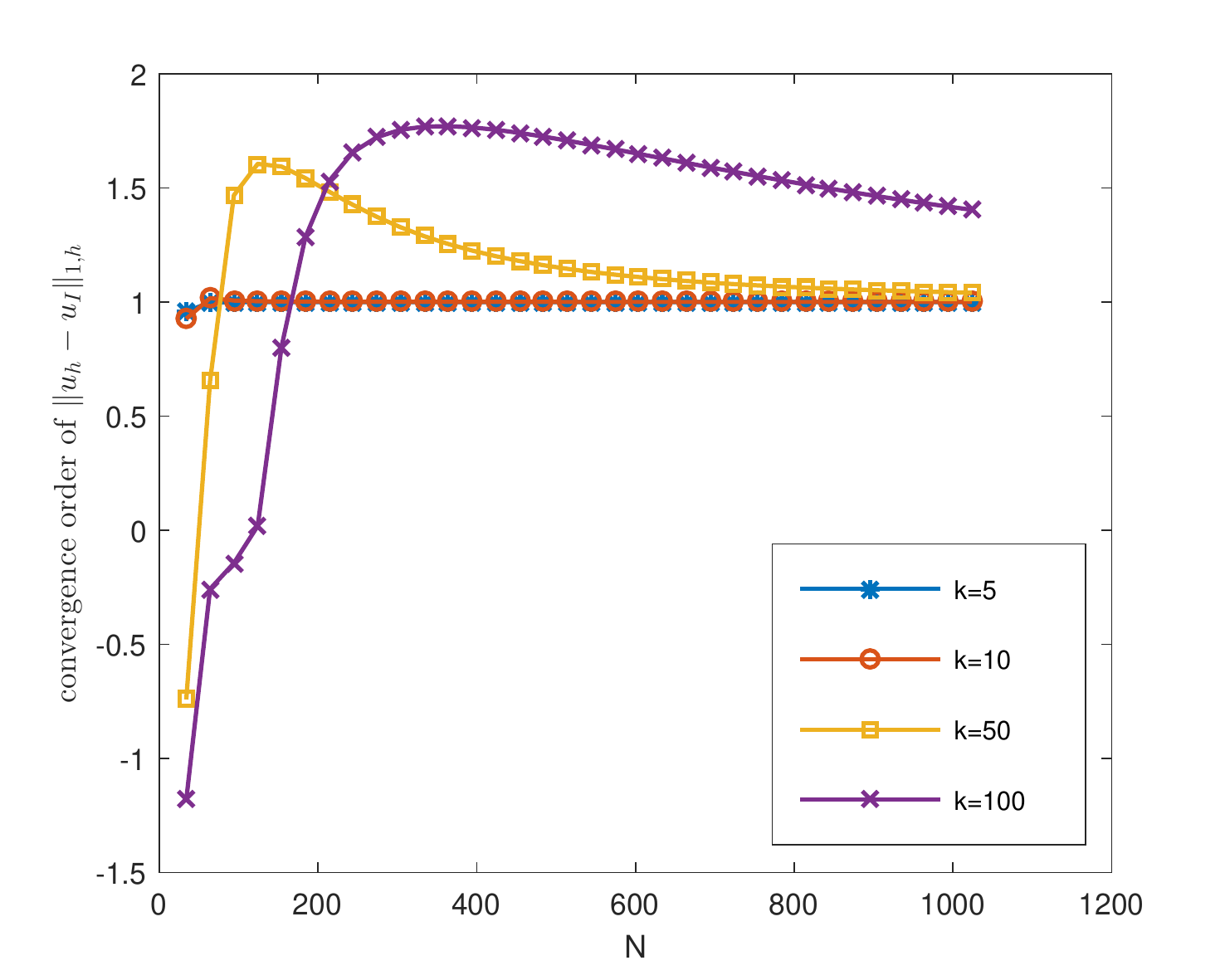}
\caption{ $\norm{u_h-u_I}_{1,h}$ (left) and the convergence order of $\norm{u_h-u_I}_{1,h}$ (right) for $k=5,10,50,100$, where $u_h$ is the numerical solution over the regular pattern uniform triangulation $\T_N$ with $\mu=0$. }
\label{fig_error1}
\end{center}
\end{figure}

\begin{figure}[htbp]
\begin{center}
\includegraphics[width=0.49\textwidth]{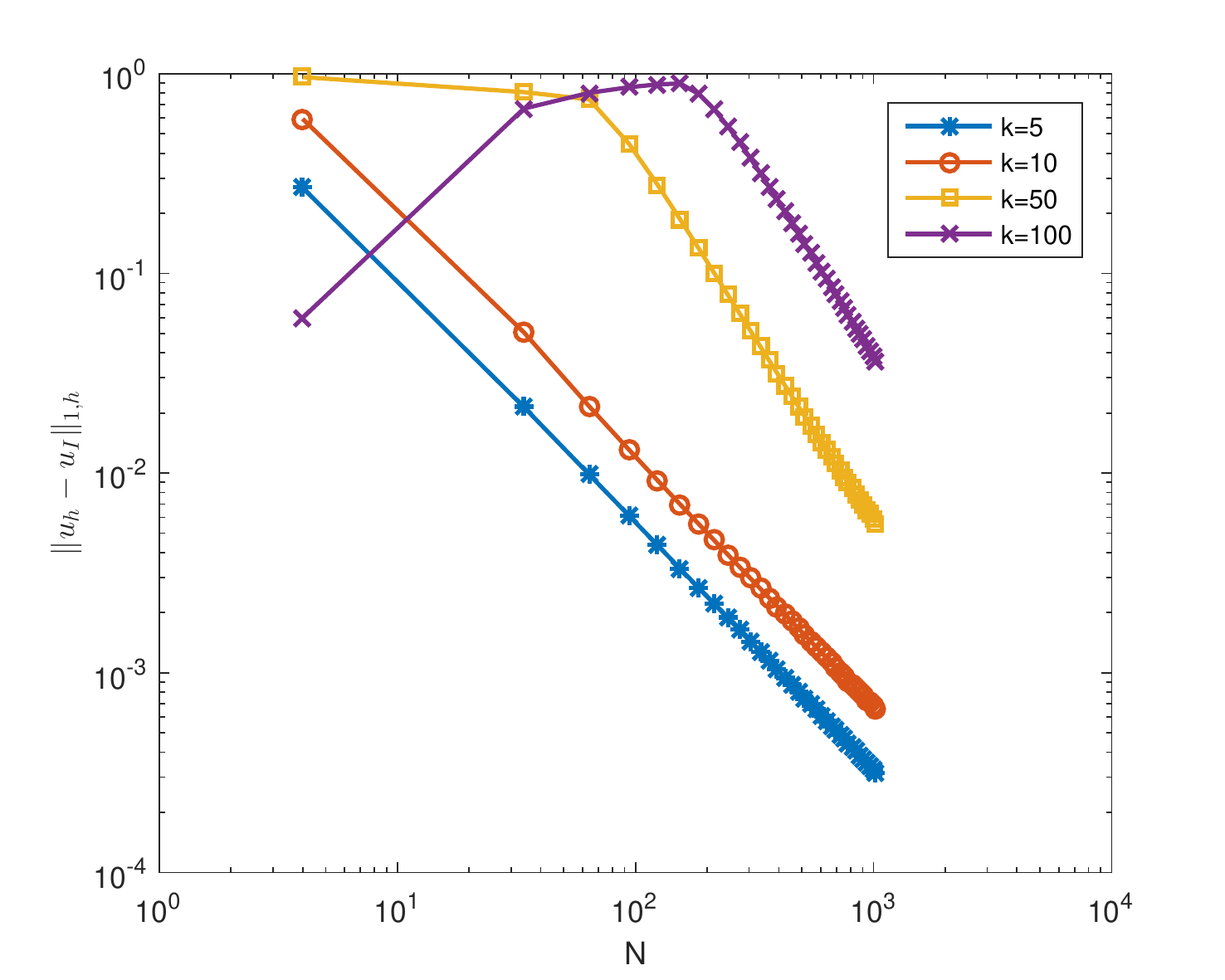}
\includegraphics[width=0.49\textwidth]{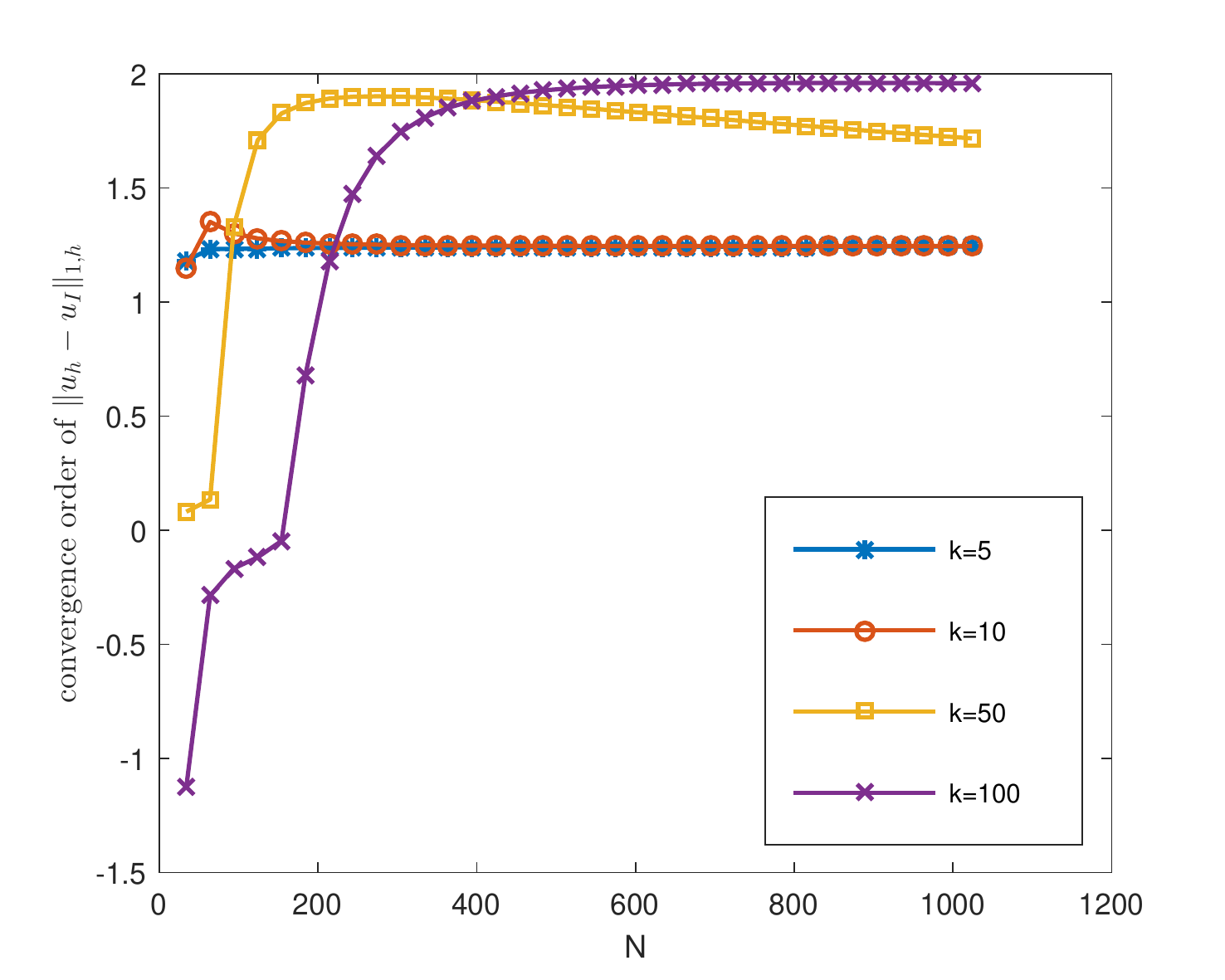}
\caption{ $\norm{u_h-u_I}_{1,h}$ (left) and the convergence order of $\norm{u_h-u_I}_{1,h}$ (right) for $k=5,10,50,100$, where $u_h$ is the numerical solution over the regular pattern uniform triangulation $\T_N$ with $\mu=1/2$. }
\label{fig_error2}
\end{center}
\end{figure}

\begin{figure}[htbp]
\begin{center}
\includegraphics[width=0.49\textwidth]{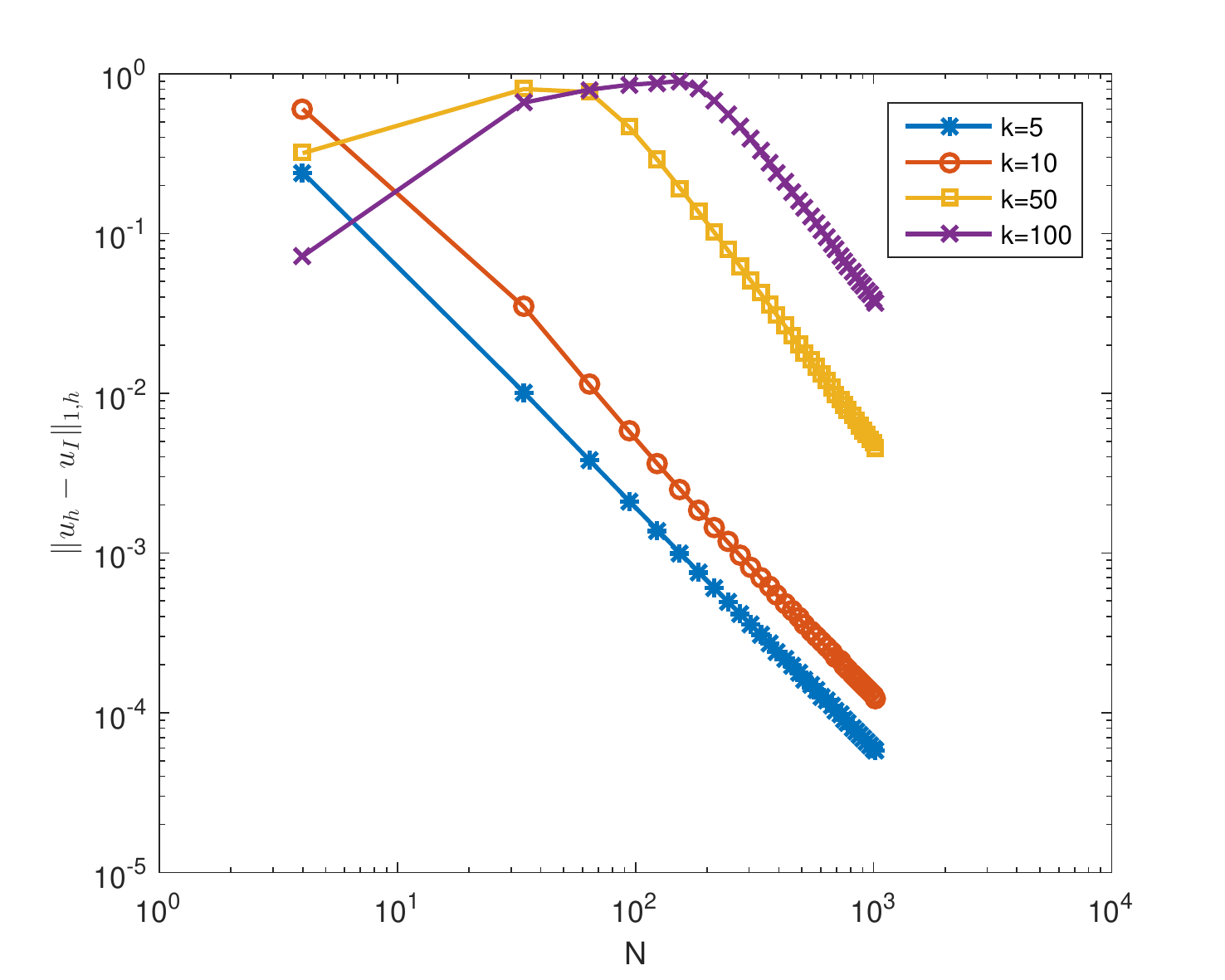}
\includegraphics[width=0.49\textwidth]{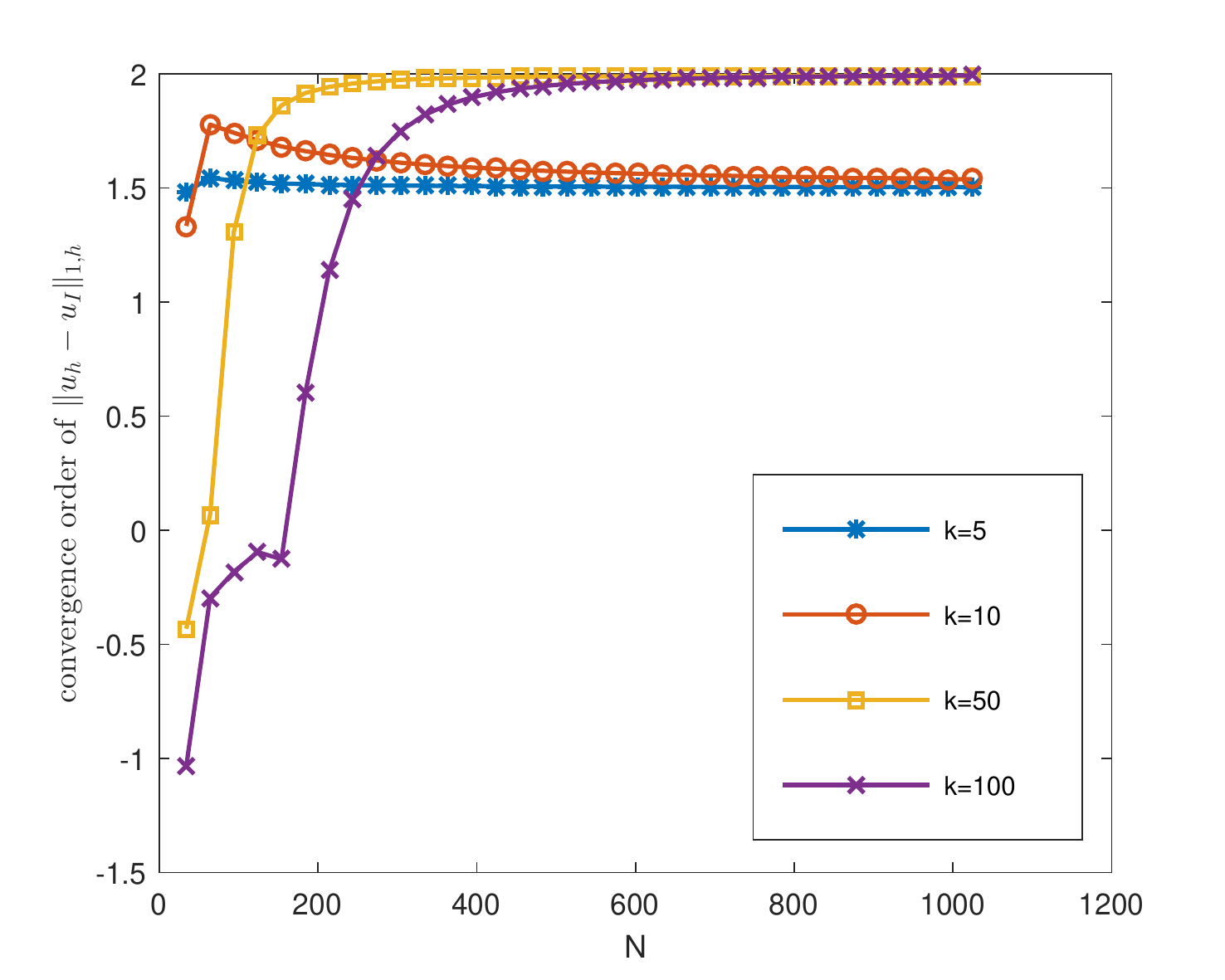}
\caption{ $\norm{u_h-u_I}_{1,h}$ (left) and the convergence order of $\norm{u_h-u_I}_{1,h}$ (right) for $k=5,10,50,100$, where $u_h$ is the numerical solution over the regular pattern uniform triangulation $\T_N$ with $\mu=1$. }
\label{fig_error3}
\end{center}
\end{figure}

\begin{figure}[htbp]
\begin{center}
\includegraphics[width=0.49\textwidth]{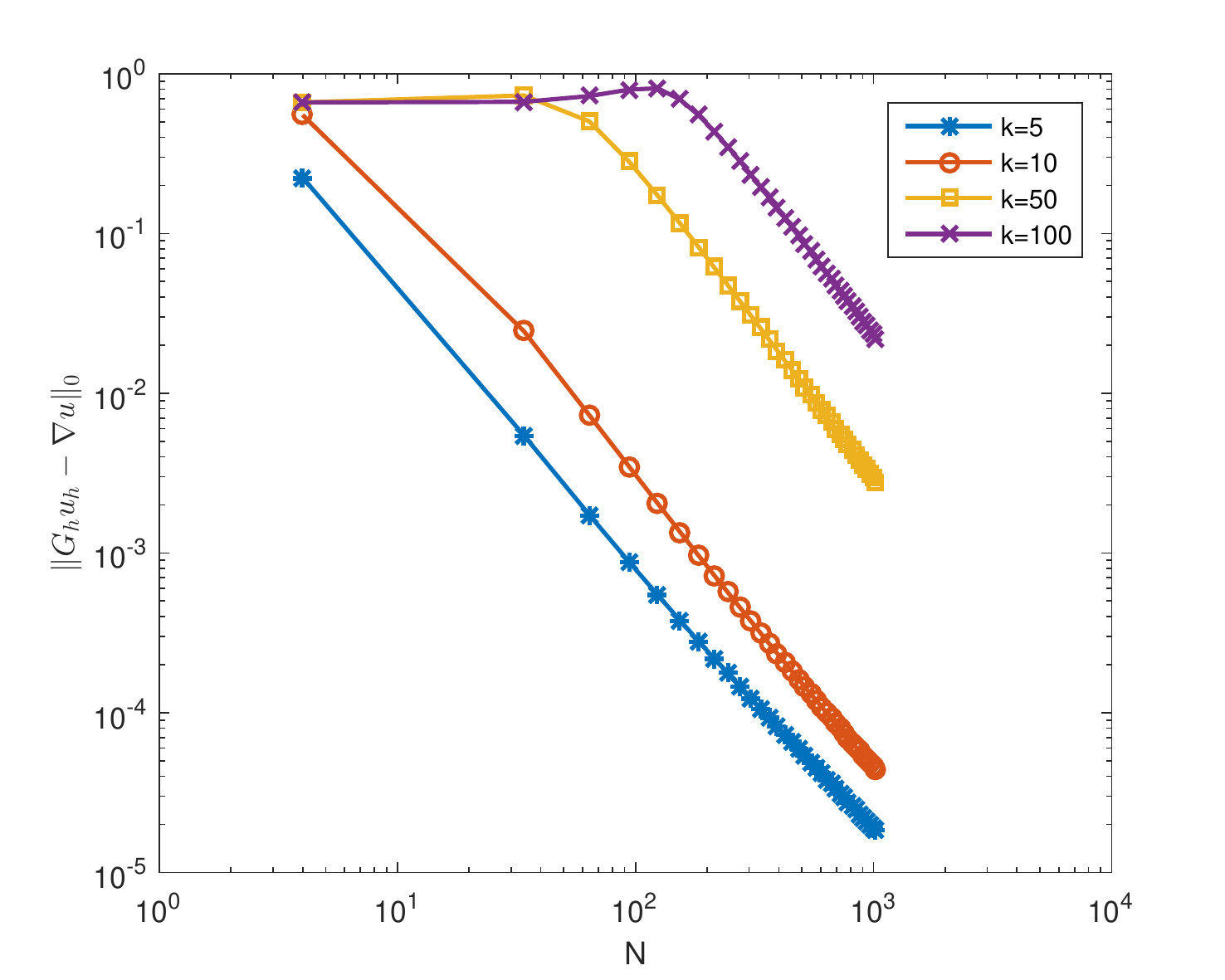}
\includegraphics[width=0.49\textwidth]{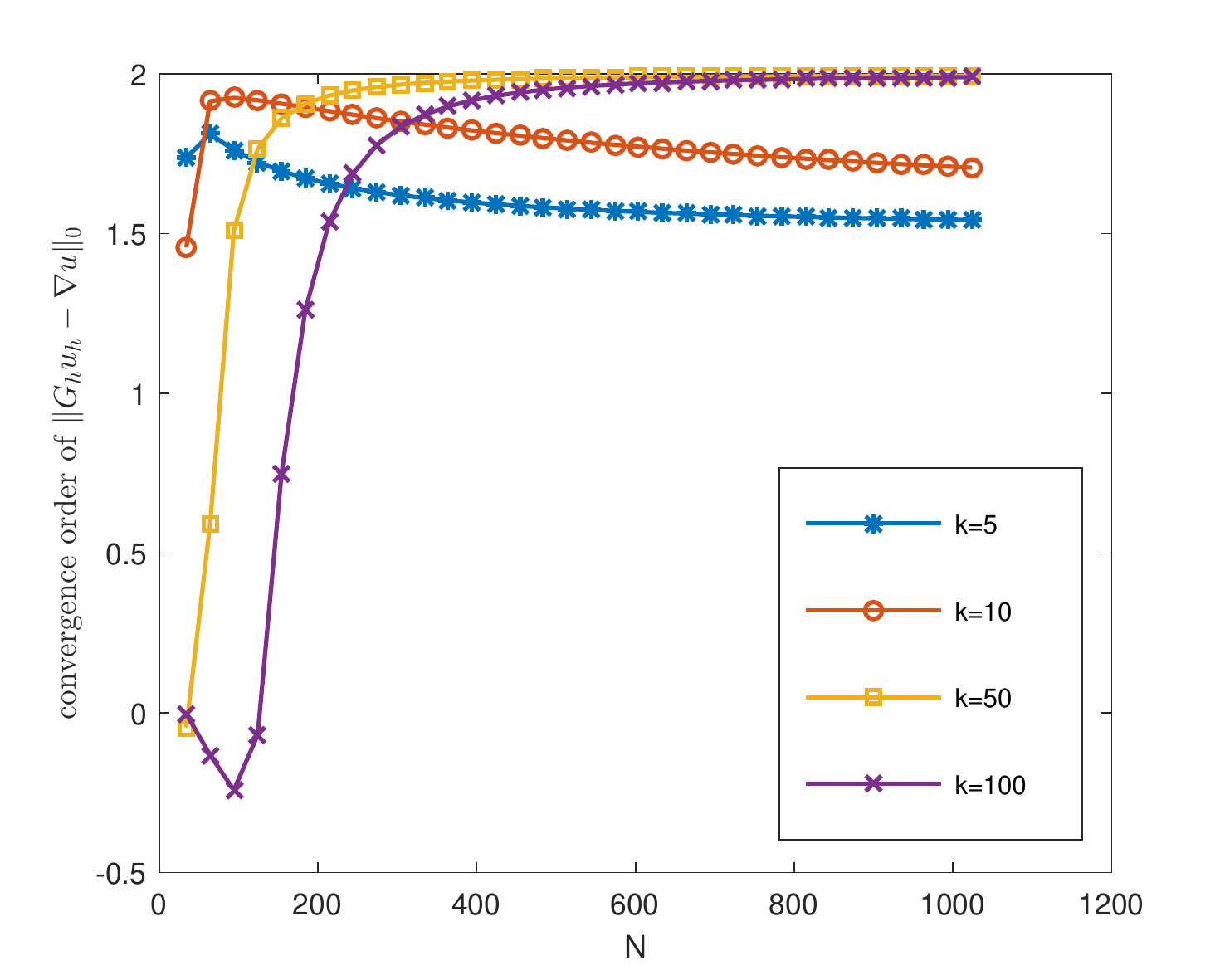}
\caption{ $\norm{G_hu_h-\nabla u}_0$ (left) and the convergence order of $\norm{G_hu_h-\nabla u}_0$ (right) for $k=5,10,50,100$, where $u_h$ is the numerical solution over the regular pattern uniform triangulation $\T_N$ with $\mu=0$.  }
\label{fig_ppr1}
\end{center}
\end{figure}

\begin{figure}[htbp]
\begin{center}
\includegraphics[width=0.49\textwidth]{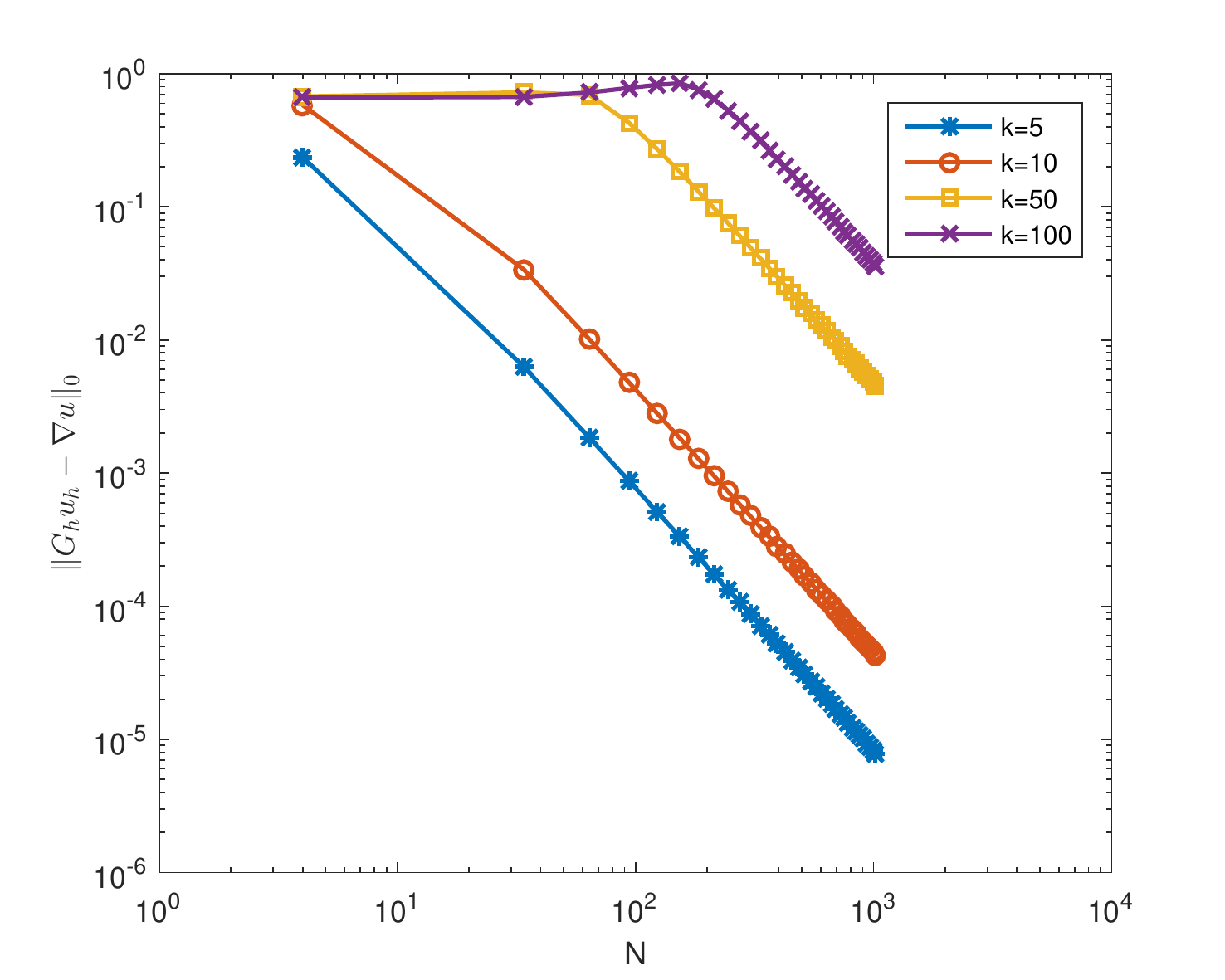}
\includegraphics[width=0.49\textwidth]{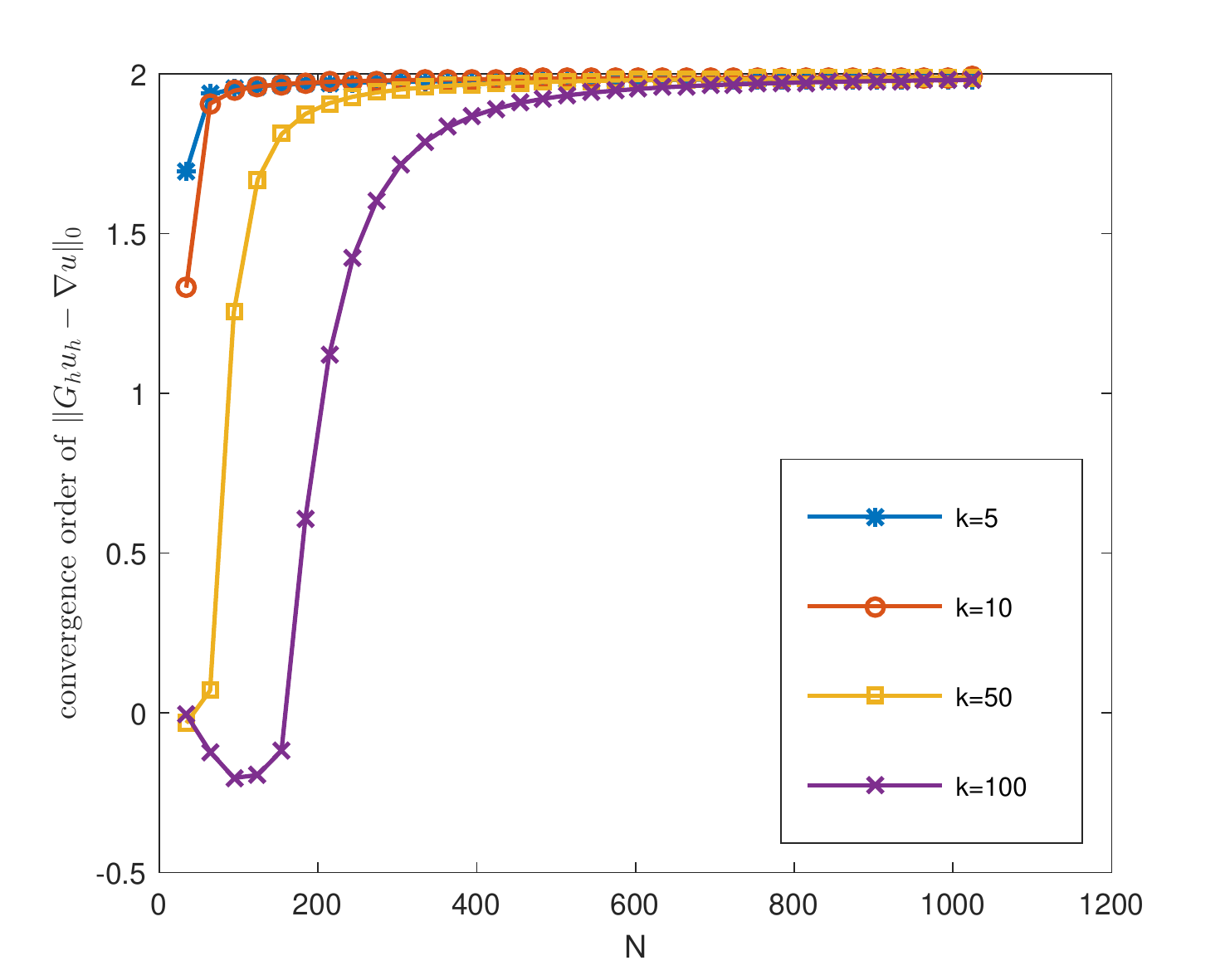}
\caption{ $\norm{G_hu_h-\nabla u}_0$ (left) and the convergence order of $\norm{G_hu_h-\nabla u}_0$ (right) for $k=5,10,50,100$, where $u_h$ is the numerical solution over the regular pattern uniform triangulation $\T_N$ with $\mu=1/2$.  }
\label{fig_ppr2}
\end{center}
\end{figure}

\begin{table}
\begin{center}
\begin{tabular}{|c|c|c|c|c|c|c|}
\hline\noalign{\smallskip}
m & \multicolumn{3}{|c|}{k=10} & \multicolumn{3}{|c|}{k=50} \\
\hline\noalign{\smallskip}
& $E_1$ & $E_2$ & $E_3$  & $E_1$ & $E_2$ & $E_3$  \\
\noalign{\smallskip}\hline\noalign{\smallskip}
4 &      6.9541e-01      &  6.5276e-01 &              &8.9283e-01&8.9323e-01& \\
8 &      4.0277e-01      &  3.0341e-01 &  2.3298e-01  &8.8133e-01&8.5489e-01&8.5858e-01\\
16 &     1.9925e-01      &  9.8125e-02 &  4.9250e-02  &9.8798e-01&8.6095e-01&8.6738e-01\\
32 &     9.7134e-02      &  2.7271e-02 &  8.9123e-03  &1.1286e+00&9.2875e-01&9.9737e-01\\
64 &     4.8052e-02      &  7.2024e-03 &  2.1001e-03  &6.8806e-01&6.0206e-01&6.6522e-01\\
128 &    2.3946e-02      &  1.8968e-03 &  5.9995e-04  &2.3326e-01&1.8737e-01&1.1690e-01\\
256 &    1.1961e-02      &  5.1132e-04 &  1.8734e-04  &8.3218e-02&4.9127e-02&1.0026e-02\\
512 &    5.9790e-03      &  1.4392e-04 &  6.1712e-05  &3.5668e-02&1.2443e-02&9.8175e-04\\
1024 &   2.9892e-03      &  4.2890e-05 &  2.1006e-05  &1.6989e-02&3.1245e-03&2.0274e-04\\
\noalign{\smallskip}\hline
\end{tabular}
\caption{ The numerical errors $E_1:=\abs{u-u_h}_{\cH^1(\T_h)}$,  $E_2:=\norm{\na u-G_h u_h}_0$ and $E_3:=\norm{\na u-RG_H u_h}_0$ with $\mu=0$ over $\T_m$ ($m=4,8,16,\ldots,1024$) for $k=10, 50$.}
\label{tab1}
\end{center}
\end{table}

\begin{table}
\begin{center}
\begin{tabular}{|c|c|c|c|c|c|c|}
\hline\noalign{\smallskip}
m & \multicolumn{3}{|c|}{k=10} & \multicolumn{3}{|c|}{k=50} \\
\hline\noalign{\smallskip}
& $E_1$ & $E_2$ & $E_3$  & $E_1$ & $E_2$ & $E_3$  \\
\noalign{\smallskip}\hline\noalign{\smallskip}
4			&	7.6897e-01	&	7.1018e-01	&				&	9.0817e-01	&	9.0317e-01	&				\\
8			&	4.6583e-01	&	3.8200e-01	&	3.2520e-01	&	8.5422e-01	&	8.5125e-01	&	8.5155e-01	\\
16			&	2.1661e-01	&	1.3692e-01	&	7.2280e-02	&	9.8422e-01	&	8.6858e-01	&	8.8154e-01	\\
32			&	9.9911e-02	&	3.8934e-02	&	9.8256e-03	&	1.0561e+00	&	9.1252e-01	&	9.7172e-01	\\
64			&	4.8407e-02	&	1.0143e-02	&	1.6968e-03	&	9.5362e-01	&	8.7354e-01	&	1.0050e+00	\\
128			&	2.3987e-02	&	2.5721e-03	&	3.9594e-04	&	3.3755e-01	&	3.0452e-01	&	2.7731e-01	\\
256			&	1.1965e-02	&	6.4685e-04	&	9.7480e-05	&	1.0473e-01	&	8.0039e-02	&	2.4874e-02	\\
512			&	5.9791e-03	&	1.6231e-04	&	2.4212e-05	&	3.9086e-02	&	2.0213e-02	&	1.7802e-03	\\
1024		&	2.9891e-03	&	4.0800e-05	&	6.0608e-06	&	1.7453e-02	&	5.0653e-03	&	1.9414e-04	\\
\noalign{\smallskip}\hline
\end{tabular}
\caption{ The numerical errors $E_1:=\abs{u-u_h}_{\cH^1(\T_h)}$,  $E_2:=\norm{\na u-G_h u_h}_0$ and $E_3:=\norm{\na u-RG_H u_h}_0$ with $\mu=1$ over $\T_m$ ($m=4,8,16,\ldots,1024$) for $k=10, 50$.}
\label{tab2}
\end{center}
\end{table}

\begin{table}
\begin{center}
\begin{tabular}{|c|c|c|c|c|c|c|}
\hline\noalign{\smallskip}
m & \multicolumn{3}{|c|}{k=10} & \multicolumn{3}{|c|}{k=50} \\
\hline\noalign{\smallskip}
& $E_1$ & $E_2$ & $E_3$  & $E_1$ & $E_2$ & $E_3$  \\
\noalign{\smallskip}\hline\noalign{\smallskip}
4 	&	8.1033e-01	&	7.4112e-01	&		&	8.9147e-01	&	8.9225e-01	&		\\
8 	&	4.8179e-01	&	4.0019e-01	&	3.4405e-01	&	8.5410e-01	&	8.5123e-01	&	8.5150e-01	\\
16 	&	2.1893e-01	&	1.4118e-01	&	7.2934e-02	&	9.8776e-01	&	8.6974e-01	&	8.8355e-01	\\
32 	&	1.0011e-01	&	3.9575e-02	&	9.4564e-03	&	1.0569e+00	&	9.1404e-01	&	9.7362e-01	\\
64 	&	4.8420e-02	&	1.0227e-02	&	1.6651e-03	&	9.5842e-01	&	8.7845e-01	&	1.0097e+00	\\
128 	&	2.3987e-02	&	2.5827e-03	&	3.9469e-04	&	3.3887e-01	&	3.0595e-01	&	2.8075e-01	\\
256 	&	1.1965e-02	&	6.4817e-04	&	9.7438e-05	&	1.0488e-01	&	8.0231e-02	&	2.5034e-02	\\
512 	&	5.9791e-03	&	1.6221e-04	&	2.4213e-05	&	3.9099e-02	&	2.0237e-02	&	1.7813e-03	\\
1024 	&	2.9891e-03	&	3.6859e-05	&	9.2329e-06	&	1.7454e-02	&	5.0672e-03	&	1.9374e-04	\\
\noalign{\smallskip}\hline
\end{tabular}
\caption{ The numerical errors $E_1:=\abs{u-u_h}_{\cH^1(\T_h)}$,  $E_2:=\norm{\na u-G_h u_h}_0$ and $E_3:=\norm{\na u-RG_H u_h}_0$ with $\mu=2$ over $\T_m$ ($m=4,8,16,\ldots,1024$) for $k=10, 50$.}
\label{tab3}
\end{center}
\end{table}

\begin{table}
\begin{center}
\begin{tabular}{|c|c|c|c|c|c|c|}
\hline\noalign{\smallskip}
m & \multicolumn{2}{|c|}{k=10} & \multicolumn{2}{|c|}{k=60} & \multicolumn{2}{|c|}{k=120}\\
\hline\noalign{\smallskip}
& $E_1$ & $\eta_h$ & $E_1$  & $\eta_h$ & $E_1$ & $\eta_h$  \\
\noalign{\smallskip}\hline\noalign{\smallskip}
4 	&	6.95e-01	&		&	8.35e-01	&		&	8.22e-01	&		\\
8 	&	4.03e-01	&	3.76e-01	&	9.25e-01	&	2.22e-01	&	8.51e-01	&	3.36e-02	\\
16 	&	1.99e-01	&	1.96e-01	&	9.95e-01	&	5.70e-01	&	9.16e-01	&	2.33e-01	\\
32 	&	9.71e-02	&	9.71e-02	&	1.06e+00	&	4.93e-01	&	9.89e-01	&	5.26e-01	\\
64 	&	4.81e-02	&	4.81e-02	&	9.80e-01	&	3.50e-01	&	1.04e+00	&	4.47e-01	\\
128 	&	2.39e-02	&	2.40e-02	&	3.63e-01	&	2.59e-01	&	1.19e+00	&	3.16e-01	\\
256 	&	1.20e-02	&	1.20e-02	&	1.18e-01	&	1.12e-01	&	6.71e-01	&	3.53e-01	\\
512 	&	5.98e-03	&	5.98e-03	&	4.57e-02	&	4.54e-02	&	1.95e-01	&	1.79e-01	\\
1024 	&	2.99e-03	&	2.99e-03	&	2.08e-02	&	2.08e-02	&	6.05e-02	&	5.98e-02	\\
\noalign{\smallskip}\hline
\end{tabular}
\caption{ The numerical errors $E_1:=\abs{u-u_h}_{\cH^1(\T_h)}$ and $\eta_h$ with $\mu=0$ over $\T_m$ ($m=4,8,16,\ldots,1024$) for $k=10, 60, 120$.}
\label{tab4}
\end{center}
\end{table}

\begin{table}
\begin{center}
\begin{tabular}{|c|c|c|c|c|c|c|}
\hline\noalign{\smallskip}
m & \multicolumn{2}{|c|}{k=10} & \multicolumn{2}{|c|}{k=60} & \multicolumn{2}{|c|}{k=120}\\
\hline\noalign{\smallskip}
& $E_1$ & $\eta_h$ & $E_1$  & $\eta_h$ & $E_1$ & $\eta_h$  \\
\noalign{\smallskip}\hline\noalign{\smallskip}
4 	&	7.69e-01	&		&	8.42e-01	&		&	8.22e-01	&		\\
8 	&	4.66e-01	&	3.57e-01	&	8.87e-01	&	5.44e-02	&	8.55e-01	&	8.54e-02	\\
16 	&	2.17e-01	&	1.96e-01	&	9.29e-01	&	4.30e-01	&	8.78e-01	&	3.58e-02	\\
32 	&	9.99e-02	&	9.84e-02	&	1.02e+00	&	4.07e-01	&	9.30e-01	&	3.89e-01	\\
64 	&	4.84e-02	&	4.84e-02	&	1.14e+00	&	3.09e-01	&	1.01e+00	&	3.70e-01	\\
128 	&	2.40e-02	&	2.40e-02	&	5.43e-01	&	3.19e-01	&	1.11e+00	&	2.77e-01	\\
256 	&	1.20e-02	&	1.20e-02	&	1.61e-01	&	1.49e-01	&	9.68e-01	&	3.36e-01	\\
512 	&	5.98e-03	&	5.98e-03	&	5.35e-02	&	5.30e-02	&	3.00e-01	&	2.60e-01	\\
1024 	&	2.99e-03	&	2.99e-03	&	2.19e-02	&	2.19e-02	&	8.35e-02	&	8.24e-02	\\
\noalign{\smallskip}\hline
\end{tabular}
\caption{ The numerical errors $E_1:=\abs{u-u_h}_{\cH^1(\T_h)}$ and $\eta_h$ with $\mu=1$ over $\T_m$ ($m=4,8,16,\ldots,1024$) for $k=10, 60, 120$.}
\label{tab5}
\end{center}
\end{table}

\bibliographystyle{siam}
\bibliography{referrence}
\end{document}